\documentclass{amsart}
\usepackage{amssymb}
\usepackage{amsmath}
\usepackage{color}
\usepackage{epsfig}
\newtheorem{mainthm}{Theorem}[]\newtheorem{maincor}[mainthm]{Corollary}
\newtheorem{thm}{Theorem}[section]
\newtheorem*{prob}{Problem}

\newtheorem*{thm*}{Theorem}

\newtheorem{lem}[thm]{Lemma}

\newtheorem{exa}[thm]{Example}
\newtheorem*{exam}{Example}

\newtheorem{prop}[thm]{Proposition}

\theoremstyle{definition}

\newtheorem{defin}[thm]{Definition}
\newtheorem*{rems*}{Remarks}
\theoremstyle{remark}
\newtheorem{rem}[thm]{Remark}

\newcommand{\CP}{\mathbb{C\mkern1mu P}}               % complex proj. space
\newcommand{\C}{\mathbb{C}}

\newcommand{\N}{\mathbb{N}}

\newcommand{\R}{\mathbb{R}}

\newcommand{\Z}{\mathbb{Z}}

\newcommand{\si}{\sigma}
\newcommand{\ti}{\tilde}

\newcommand{\subsub}{{\subset \subset} }

\DeclareMathOperator{\tr}{tr}

\DeclareMathOperator{\Ric}{Ric}\DeclareMathOperator{\ad}{ad}

\DeclareMathOperator{\scal}{scal}

\DeclareMathOperator{\Or}{O}

\DeclareMathOperator{\Rm}{Rm}
\DeclareMathOperator{\Wc}{W}

\newcommand{\lap}{\Delta}
\newcommand{\ep}{\epsilon}
\newcommand{\al}{\alpha}

\newcommand{\de}{\delta}

\newcommand{\gl}{\mathfrak{gl}}\newcommand{\su}{\mathfrak{su}}
\newcommand{\un}{\mathfrak{u}}

\newcommand{\Lg}{\mathfrak{g}}
 \newcommand{\Lm}{\mathfrak{m}}

\newcommand{\Lh}{\mathfrak{h}}

\newcommand{\Ln}{\mathfrak{n}}

\newcommand{\jkm}{[jkm]}

\DeclareMathOperator{\trace}{tr}
\DeclareMathOperator{\vol}{vol}\DeclareMathOperator{\Ad}{Ad}

\DeclareMathOperator{\dist}{dist}

\DeclareMathOperator{\id}{{id}} 
\DeclareMathOperator{\Weyl}{W}

\newcommand{\of}{\circ}

\newcommand{\eps}{\varepsilon}

\newcommand{\norm}{ \| }

\newcommand{\TS}{\mathcal{S}}
\newcommand{\unm}{\tfrac{1}{2}}
\begin{document}

%\dedicatory{}
%\subjclass{}
%\keywords{diameter rigidity, positive curvature,
%indices of closed geodesics, $P_l$-manifolds}
%\thanks{}
%\thanks{The first author was supported by the Deutsche
%Forschungsgemeinschaft}
%\tableofcontents

\begin{titlepage}

\title{Optimal curvature estimates for homogeneous Ricci flows}

\author{Christoph B\"ohm}	
\address{University of M\"unster, Einsteinstra{\ss}e 62, 48149 M\"unster, Germany}
\email{cboehm@math.uni-muenster.de}

\author{Ramiro Lafuente} 
\address{University of M\"unster, Einsteinstra{\ss}e 62, 48149 M\"unster, Germany}
\email{lafuente@uni-muenster.de}

\author{Miles Simon}
\address{Otto von Guericke University, Magdeburg, IAN, Universit\"atsplatz 2, Magdeburg 39104, Germany}
\email{msimon@gmx.de}

% \author{Christoph B\"ohm, Ramiro Lafuente and Miles Simon}
%\address{University of Pennsylvania\\ DRL\\
%209 South 33rd  Street \\
%     Philadelphia, PA 19104-6395\\  USA}
%\email{wilking@math.upenn.edu}
\begin{abstract}
$\!$We prove uniform curvature estimates for homogeneous Ricci flows: For a solution defined on  $[0,t]$ the norm of the curvature tensor at time $t$ is bounded by the maximum of
$C(n)/t$ and $C(n)(\scal(g(t)) - \scal(g(0)) )$.
%, where $\scal$ denotes scalar curvature. 
This is used to show that
solutions with finite extinction time are Type I,
immortal solutions are Type III and ancient solutions are Type I, where all the constants involved depend  only on the dimension $n$. 
A further consequence is that a non-collapsed homogeneous ancient solution
on a compact homogeneous space emerges from a unique Einstein metric on the same space.

The above curvature estimates are proved using a gap theorem for Ricci-flatness on
homogeneous spaces. 
The proof of this gap theorem is by contradiction  and uses a local $W^{2,p}$ convergence result, which holds without symmetry assumptions.

\end{abstract}

\end{titlepage}

\maketitle
\setcounter{page}{1}
\setcounter{tocdepth}{0}

The proof of Thurston's geometrization conjecture by  Perelman
\cite{Per1}, \cite{Per2}, \cite{Per3} using Hamilton's Ricci  flow \cite{Ha} can
certainly be considered a major break through. There are however interesting related problems which remain open. 
For instance, Lott asked in \cite{Lo2},
whether the  $3$-dimensional Ricci flow detects the homogeneous pieces 
in the geometric decomposition proposed by Thurston.
In the same paper this was proved  to be true for immortal solutions,
assuming a \mbox{Type III}
 behavior of the curvature tensor and a natural bound on the
diameter of the underlying closed oriented manifold. More recently, Bamler 
showed in a series of papers that for the  Ricci flow 
with surgery there exist only finitely many surgery times, 
and that the Type III behavior holds after the last surgery time. 
In many cases, convergence to a geometric piece could be established.
We refer to \cite{Bam} and the papers quoted therein.

Recall that a Ricci flow solution is called \emph{homogeneous}, if it is homogeneous
at every time. In dimension $3$ homogeneous Ricci flows
are well understood: see \cite{IJ}, \cite{KM}, \cite{CSC},  \cite{GP}, \cite{Lo1}. For results in higher dimensions we refer to \cite{IJL}, \cite{BW}, \cite{AN}, \cite{Bu}, \cite{Pa}, \cite{AC}, \cite{L11}, \cite{L13}, \cite{Ar}, among others. Notice that except for \cite{L13}, assumptions on the algebraic structure or on the dimension were made. 

Our first main result is 

\begin{mainthm}\label{main:long}
Let $\left(M^n, g(t)\right)_{t\in[a,b]}$ be a homogeneous Ricci flow solution. Then
the norm of the Riemannian curvature tensor at the final time $b$ can be estimated by
 \begin{eqnarray*}
 	\Vert {\Rm(g(b)) } \Vert_{g(b)} &\leq &  
 	 C(n)\cdot \max \big\{ \,\tfrac{1}{b-a}\,\,,\,\, \scal(g(b)) - \scal(g(a))\, \big\}\,.
 \end{eqnarray*}
\end{mainthm}

%Here and in that which follows, 
Symbols like $c(n),C(n)$, etc.~refer to positive constants which depend only on the dimension. 
A first immediate consequence of the above estimate is that for homogeneous Einstein spaces of a fixed dimension, the Einstein constant controls the norm of the curvature tensor. Notice that this is not true already in the case of cohomogeneity one Einstein spaces with positive Einstein constant \cite{Bo1}.

\begin{maincor}\label{main:cor}
Let $\left(M^n, g(t)\right)_{t\in I}$ be a homogeneous Ricci flow solution. Then
the following holds: If the solution has finite extinction time $T$, i.e.~ $I=[0,T)$, and $\scal(g(0))=1$, then there exists $\delta(n) \in (0,1)$ such that
for all $t \in [\delta(n)\cdot T,T)$
 \[
 	\Vert {\Rm(g(t))} \Vert_{g(t)} \cdot (T-t) \,\,\, \in \,\,\,	[\tfrac{1}{8},C(n)]\,.
 \]
	If the solution is immortal, i.e.~ $I=[0,\infty)$, and $\scal(g(0))=-1$, then for all $t \in I$
	\[
	\Vert {\Rm(g(t))} \Vert_{g(t)}\cdot t  \,\,\, \in \,\,\,  [0,C(n)]\,.
	\] 
If the solution is ancient, i.e.~ $I=(-\infty,-1]$, and $\scal(g(-1))=1$, then for all 
$t \in I$
	\[
	\Vert {\Rm(g(t))} \Vert_{g(t)}\cdot \vert t \vert   \,\,\, \in \,\,\, [c(n),  C(n) ]\,. 
	\] 
\end{maincor}

The above assumptions on the scalar curvature can always be achieved: see Remark \ref{rem_ass}. Note also, that in \cite{Bo3} the first two upper curvature bounds were shown, however with constants depending on the initial metric. 

Regarding homogeneous solutions with finite extinction time, recall that by \cite{BB} a homogeneous space always admits such solutions, if its universal cover is not diffeomorphic to Euclidean space, and that starting from dimension $7$ there exists infinitely many homotopy types of simply-connected homogeneous spaces \cite{AW}. Notice also, that  the homogeneity assumption cannot be dropped in the above corollary, see \cite{Ki}, \cite{DH}, \cite{GZ}. Moreover,
we show in  Lemma \ref{lem_s3} that there cannot exist a uniform upper bound for $\Vert \Rm(g(t))\Vert_{g(t)} \cdot (T-t)$ for small times: On $S^3$ there exists a sequence of homogeneous Ricci flows, such that at 
the initial time the norm of the curvature tensor is one,
the scalar curvature is positive, but the extinction times are unbounded.

Using the above corollary, it then follows from \cite{Na} and \cite{EMT} that
a homogeneous Ricci flow solution with finite extinction time
subconverges, after appropriate scaling, 
to a non flat homogeneous gradient 
shrinking soliton. By \cite{PW}, such a forward limit soliton is a
 finite quotient of a product of a compact homogeneous 
Einstein space and a non-compact flat factor, where the latter might be absent.

Next, recall that by \cite{BB} and \cite{La} \emph{any} homogeneous Ricci flow on a homogeneous space whose universal cover is diffeomorphic to Euclidean space is immortal and that starting in dimension $3$ there are uncountably many homogeneous spaces whose underlying manifold is Euclidean space \cite{Bi}.
For immortal solutions the lower \mbox{bound $0$} in the above estimate is again
optimal: already in dimension three, there are examples where the curvature tensor converges exponentially fast to zero \cite{IJ}.
Concerning forward limit non-gradient solitons of immortal homogeneous solutions we refer to \cite{BL}.

We turn to homogeneous ancient solutions. 
Notice first, that the homogeneity assumption cannot be dropped in the above corollary in view of \cite{Per2}, \cite{BKN}. 
Recall also that homogeneous ancient solutions are called non-col\-lapsed, if the corresponding curvature normalized metrics have a uniform lower injectivity radius bound. In this case, by the above corollary and by \cite{Na}, \cite{CZ} these solutions admit a non-flat homogeneous asymptotic soliton as one goes backwards in time.

All known non-compact ancient homogeneous Ricci flow solutions
are the Riemannian product of a compact ancient solution and a flat factor. In the compact case our  estimates yield the following

\begin{mainthm}\label{main:comp}
The asymptotic soliton of a non-collapsed, homogeneous ancient solution 
on a compact homogeneous space is compact and unique.
\end{mainthm}

In fact, we show that non-collapsed ancient solutions on a compact homogeneous space must emanate from a homogeneous Einstein metric on the same space. Examples of such solutions have been described in \cite{BKN}, \cite{Bu}, both with compact and non-compact forward limit soliton. Let us also mention that on a compact homogeneous space which is not a homogeneous torus bundle, ancient homogeneous solutions are non-collapsed: see Remark \ref{rem_noncoll}.
 
Since along the volume-normalized Ricci flow
the scalar curvature is non-decrea\-sing, the Einstein metric from which an ancient solution emanates cannot be
a local maximum of the total scalar curvature functional  restricted
to the space of homogeneous metrics. Conversely, if an Einstein metric is not a local maximum in this sense, there exists an ancient solution emanating from it: see 
Lemma \ref{lem:ancient}.
%Section \ref{sec:comancient}. 

Next, we turn now to collapsed homogeneous ancient solutions on compact homogeneous spaces.
Since they are collapsed, the asymptotic soliton can only exist in the sense of Riemannian groupoids, as introduced  by Lott \cite{Lo1}: see 
\mbox{Section \ref{sec:exancient}}.
A nice example is given by the Berger metrics on $S^{2n+1}$. They have the non-compact asymp\-totic soliton \mbox{$\CP^{n} \times \R$}, and the round sphere as a compact forward limit soliton.

The following compact homogeneous space is the first example admitting a collapsed ancient solution with non-compact forward limit soliton. Moreover, this example also shows that the geometry of the asymptotic soliton does not depend continuously on ancient solutions.

\begin{exam}
There exists a compact homogeneous space $M^{12}$ which admits a one-parameter family of homogeneous ancient solutions with the same asymptotic soliton $(E^{11}, g_{E_1})\times \R$ and the same forward limit soliton $S^3\times \R^9$. Moreover, in the closure of these solutions there is a single ancient solution emanating from $(E^{11}, g_{E_2})\times \R$.
\end{exam}

 Here, $g_{E_1}$, $g_{E_2}$ are non-isometric Einstein metrics on the compact homogeneous space $E^{11}$. In appropriate coordinates these solutions are depicted in Figure \ref{pic:ancient}.
\begin{figure}[h]
\begin{center}
\includegraphics[width=3.5in,height=2in]{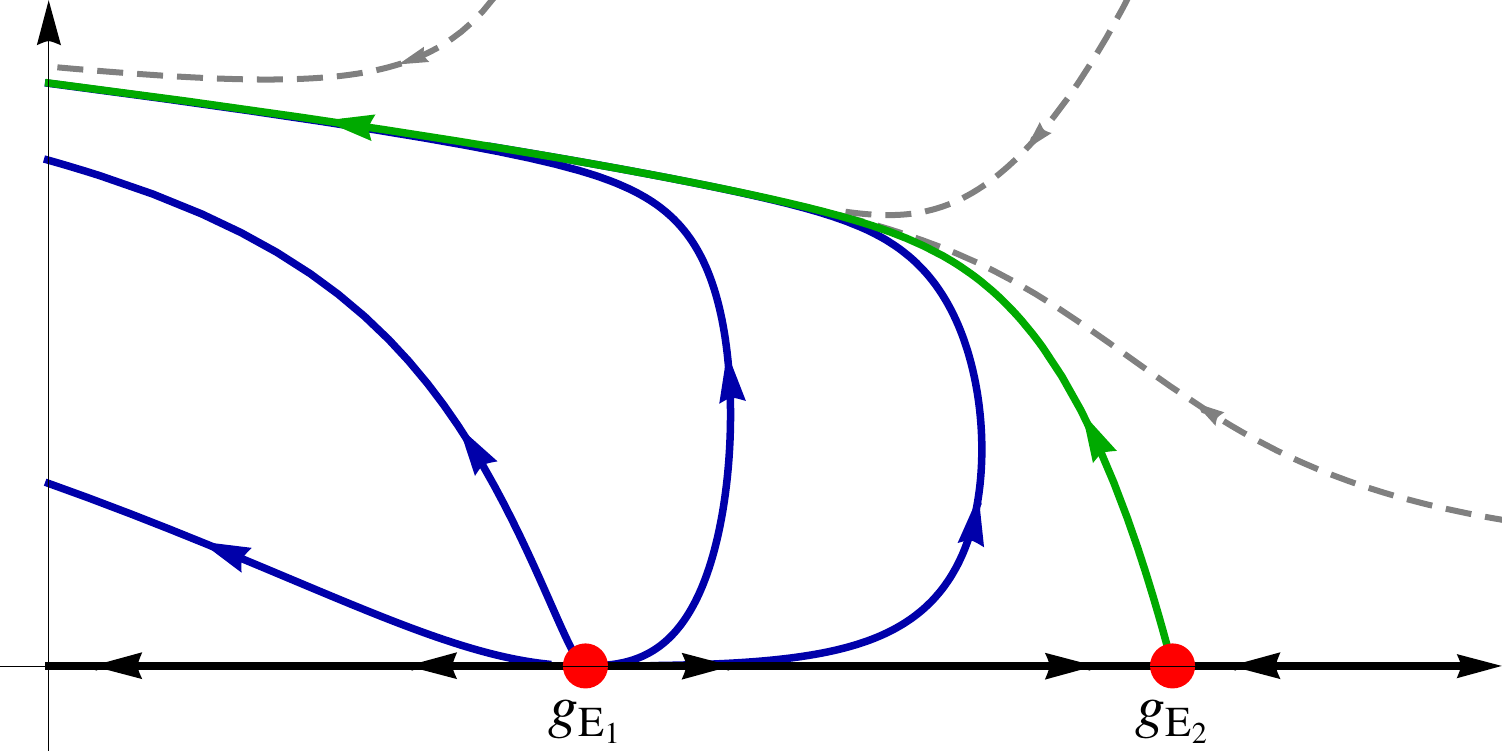}
\end{center}
\caption{A family of ancient solutions on $M^{12}$}
\label{pic:ancient}
\end{figure}
For further details and higher dimensional examples see Section \ref{sec:exancient}.

Our second main result, which is crucial for
the proof of Theorem \ref{main:long}, is

\begin{mainthm}[Gap Theorem]\label{main:gap}
There exists $\epsilon(n) \in (0,1)$ such that for any homogeneous space $(M^n,g)$ the Weyl curvature can be estimated by
\[
             \Vert {\Wc(g) } \Vert_g \leq \big( 1-\epsilon(n)  \big) \cdot \Vert {\Rm(g) }\Vert_g\,.
\]
\end{mainthm}

It follows that a  homogeneous Ricci flat space is flat, a result which was proved
by Alekseevskiĭ and Kimel{'}fel{'}d \cite{AK} in 1975. But it also shows
that a non-flat homogeneous space cannot be ``too'' Ricci flat.
The optimal gap size $\epsilon(n)$ is unknown, but converges to $0$ as $n\to \infty$: see Section \ref{sec:ex}.
It is worthwhile mentioning that the Gap Theorem is equivalent to the statement
\[
           \Vert \Rm(g) \Vert_g \leq C(n) \cdot  \Vert \Ric(g) \Vert_g\,. 
\]

The Gap Theorem is proved by contradiction.
We show that a contradiction sequence
subconverges locally in $C^{1,\alpha}$-topology to a smooth local
limit space,
when assuming norm-normalized curvature tensors. Such a local limit is a smooth Ricci flat metric.
Since it is also locally homogeneous, it must be flat \cite{Sp}.
On the other hand, as already remarked
by Anderson \cite{And},
subconvergence can even be assumed in $W^{2,p}$-topology for some $p>n/2$,
which yields a positive lower bound for the norm of the curvature tensor.

Since the corresponding curvature estimates might be of independent
interest we state them here. We would like to mention, that
in the following theorem there are no symmetry or completeness assumptions.

\begin{mainthm}\label{main:harm}
Given $0< v\leq V$ and $p\in (n/2,\infty)$ , there exists a constant $\eps=\eps(v,V,n,p)>0$ such that the following holds.
Let $(D^n_i,g_i,x_i)_{i \in \N}$ be a sequence of smooth 
manifolds,  such that $B_1^{g_i}(x_i) $  is compactly
contained in $D^n_i$ for all $i \in \N$. Assume that 
$v r^n \leq \vol(B_r^{g_i}(x)) \leq V r^n$
for all $r \leq 1 $, for all $B_r^{g_i}(x) \subseteq B_1^{g_i}(x_i)$,
and
\begin{eqnarray*}
  \lim_{i\to\infty} \,\, \int_{B_1^{g_i}(x_i)} \Vert\Ric(g_i)\Vert^{p} d\mu_{g_i}  = 0  
\quad \mbox{ and } \quad
 \int_{ B_{1}^{g_i}(x_i)  } \Vert \Rm(g_i) \Vert^{n/2} d\mu_{g_i} \leq \eps\,.
\end{eqnarray*}
Then, for all $s \in (0,1)$, $(B_s^{g_i}(x_i), g_i, x_i)_{i\in\N}$ subconverges in the pointed $W^{2,p}$-topo\-logy to a $C^\infty$-smooth limit manifold
$(B_s^{g_\infty}(x), g_\infty, x)$, and we have
\begin{eqnarray*}
  \lim_{i\to \infty} \int_{B_s^{g_i}(x_i)}  \Vert \Rm(g_i)\Vert^{p} d\mu_{g_i}
  &=&
   \int_{B_s^{g_\infty}(x)} \Vert \Rm(g_\infty)\Vert^{p} d\mu_{g_\infty} \,.
\end{eqnarray*}
\end{mainthm}

At the moment  no algebraic proof of the Gap Theorem is known, 
not even for the fact that Ricci flat homogeneous spaces are flat. Hence
we propose the following

\begin{prob}
Provide an algebraic proof for the Gap Theorem.
\end{prob}

The paper is organized as follows: In Section \ref{sec:lochom} we prove the Gap
Theorem using Theorem \ref{main:harm}, whose proof is provided
in Section \ref{sec:conv}. In Section \ref{sec:appl}
we show how \mbox{Theorem \ref{main:long}} can be deduced from Theorem \ref{main:gap} and
we give the proof of Corollary \ref{main:cor}.
In Section \ref{sec:comancient} we prove Theorem \ref{main:comp}, and in Section \ref{sec:exancient}
and we will provide examples of homogeneous ancient solutions.
Finally, in Section \ref{sec:ex} 
examples of left-invariant metrics on solvable Lie groups are given, 
which show that the constant $\epsilon(n)$ in the Gap Theorem  
must converge to zero for $n \to \infty$.

The first author would like to thank Claude Le'Brun and Lei Ni for helpful comments.
The authors would like to thank Norman Zerg\"ange for pointing out that the weak convergence theorem (see Section 2) was not correctly stated in the first version of this paper.

%%%%%%%%%%%%%%%%%%%%%%%%%%%%%%%%%%%
%%%%%%%%%%%%%%%%%%%%%%%%%%%%%%%%%%%%
%%%%%%%%%%%%%%%%%%%%%%%%%%%%%%%%%%%%%

\section{Locally and globally homogeneous spaces}\label{sec:lochom}

In this section we define locally homogeneous spaces and show that they
are real analytic Riemannian manifolds. Moreover, we provide injectivity radius estimates
and a result which shows how to extend local isometries on simply connected domains.
All of this is used to prove Theorem \ref{thm:homconv}: there, we show that a contradiction
sequence to Theorem \ref{main:gap}, lifted to the tangent spaces, 
has to subconverge in the {$C^{1,\alpha}$-topology} to a Ricci-flat space which is locally homogeneous and thus flat. Then at the end of the section we prove Theorem \ref{main:gap}.

A Riemannian manifold $(M^n  ,g)$ is called globally homogeneous if for all points
$p,q \in M^n$ there exists an isometry $f_{p,q}$ of $(M^n,g)$ mapping $p$ to $q$. In other words, the
isometry group acts transitively on $M^n$.
%Notice that the isometry $f_{p,q}$ might not be unique. 
It is a classic result that any globally homogeneous Riemannian manifold is complete.

A Riemannian manifold $(M ^n,g)$ is called \emph{locally homogeneous} if
for all $p,q \in M^n$ there exists $\eps=\eps_{p,q} >0$, depending possibly on $p$ and $q$, 
such that $B_{\eps}(p)$ and $B_{\eps} (q)$ are isometric with the induced metric.
Here, $B_{\eps}(p)$ denotes the open $\eps$-ball around $p$ in $(M^n,g)$.
Notice that a locally homogeneous manifold is not necessarily complete. Moreover, recall that there exists (incomplete) locally homogeneous manifolds which are not locally isometric to any globally homogeneous manifold, see \cite{Kow}.

\begin{lem}\label{lem:analytic}
A locally homogeneous space $(M^n,g)$ is a real analytic Riemannian manifold.
\end{lem}

 \begin{proof}
By \cite{TVh} local homogeneity of $(M^n,g)$ is equivalent
to the existence of an Ambrose-Singer-connection $\nabla^{ *} \! \!$, in short
AS-connection. An AS-connection is a metric connection,
which has parallel torsion and parallel curvature. Now by Theorem 7.7 from Chapter VI in \cite{KN1}
it follows that both $M^n$ and the AS-connection $\nabla^*$ are real analytic. 
It remains to show that the Riemannian metric $g$ is real analytic as well.

Let $x:U \to \R^n$ be a chart from the analytic atlas of $M^n$, let $V:=x(U)$
and let $(e_1,...,e_n)$ denote the standard basis on $\R^n$.
Pulling back the metric $g\vert_U$ by $x^{-1}$ to $V$ we obtain 
an metric $\hat g=(\hat g_{jk})_{1\leq j,k \leq n}$ on $V$ with $\hat g(e_j,e_k)=\hat g_{jk}$. 
Pulling back the connection $\nabla^*$
to $V$ we get an analytic connection $\hat \nabla ^*$ on $V$. Hence the
Christoffel symbols $\hat \Gamma_{ij,k}^*:V \to \R$ of $\hat \nabla^*$ 
are real analytic, $1 \leq i,j,k \leq n$.

Recall that $\hat\nabla^*$ is a metric connection,
that is for $1 \leq i,j,k \leq n$ we have
\[
  e_i \hat g(e_j,e_k)=\hat g(\hat \nabla^*_{e_i}e_j,e_k)+\hat g(e_j,\hat \nabla^*_{e_i}e_k)\,.
\]
Let $v_0 \in V$, $v \in \R^n$ with $\Vert v\Vert_{\rm std}=1$ 
and $c_v:(-\eps,\eps)\to V\,\,;\,t \mapsto v_0 + t\cdot v$.
We set $\hat g_{jk}(t):=\hat g_{jk}(c_v(t))$ for $1\leq j,k\leq n$. Then
the vector $\hat G(t)=(\hat g_{11}(t),...,\hat g_{nn}(t))$ satisfies a linear ordinary differential equation
$\hat G'(t)=\hat A(t)\cdot \hat G(t)$, where $\hat A(t)$ is real analytic. 
Now by Theorem 10.1 in  \cite{Ta} the solution $\hat G(t)$ 
is real analytic and defined on the entire interval $(-\eps,\eps)$, that is
\[
  \hat g_{jk}(t)=\sum_{l=0}^\infty \tfrac{\hat g_{jk}^{(l)}(0)}{l!}\cdot t^l,
\] 
for all $1 \leq j,k \leq n$ and all $t \in (-\eps,\eps)$.
A computation shows that $\hat g_{jk}(0)=\hat g_{jk}(v_0)$,
$\hat g_{jk}^{(1)}(0)=\langle (\nabla \hat g_{jk})_{v_0}, v\rangle$ and
\[
  \hat g_{jk}^{(2)}(0)=\tfrac{d}{dt}\vert_{t=0}\left\langle (\nabla \hat g_{jk}(c_v(t))),v\right\rangle
   =\big({\rm Hess} (\hat g_{jk})\big)_{v_0}(v,v)\,.
\]
Inductively we get corresponding formulae for the higher derivatives. This shows
that at the points $c_v(t)$ the functions $\hat g_{jk}$ can be written as a power series.
Since this works for any $v$ with $\Vert v\Vert_{\rm std}=1$
we deduce that the metric coefficients $\hat g_{jk}$ are real analytic. This shows the claim.
\end{proof}

Next, we consider globally homogeneous spaces $(M^n,g)$ with sectional curvature 
bound
$\vert K_g \vert \leq 1$ at one and hence any point of $M^n$.
We consider also for a point $p\in M^n$ the Riemannian exponential map
$\exp_p:T_pM^n \to M^n$. Since $|K_g| \leq 1$, by the Rauch comparison theorems
\[
  \exp_p\vert_{\hat B_\pi(0_p)}:\hat B_\pi(0_p) \to B_\pi(p)
\]
is an immersion. Hence we can pull back the metric $g\vert_{B_\pi(p)}$
to a metric $\hat g$ on $\hat B_\pi(0_p)\subset T_pM^n$. The metric
$\hat g$ is locally homogeneous, clearly incomplete, but still real
analytic by Lemma \ref{lem:analytic}. Here we used that the
exponential map is analytic: see Proposition 10.5 in \cite{He}.

\begin{defin}
We call $(\hat B_\pi(0_p),\hat g)$ a \emph{geometric model} for the globally homogeneous space $(M^n,g)$.
\end{defin}

We mention here  that any homogeneous space has an associated {\it infinitesimal model}
which encodes the algebraic data of the space: see \cite{TVh}.

Using once again that 
$\vert K_{\hat g}\vert \leq 1$, we see that ${\rm inj}_{\hat g}(x)
\geq i(r)>0$ for any $x \in \hat B_\pi(0_p)$
in view of \cite{CLY} or \cite{CGT}. Here
$r=d_{\hat g}(0_p,x)=\Vert x\Vert$ and $i:[0,\pi)\to \R_+$ is an explicit continuous function
with $i(r) \leq \pi-r$. Notice that $i$ does not depend on the particular choice of the 
local model $(\hat B_\pi(0_p),\hat g)$.

For the convenience of the reader we will provide a proof of a much stronger estimate.

\begin{lem}\label{lem:injest}
If  $(\hat B_\pi(0_p),\hat g)$ is a geometric model and $x \in \hat B_\pi(0_p)$ with
$r=d_{\hat g}(0_p,x)$, then $i(r)= \pi -r$.
\end{lem}
\begin{proof}
Let $x \in \hat B_\pi(0_p)$ be given and suppose that $\eps:={\rm inj}_{\hat g}(x)<\pi -r$.
Using the  triangle inequality  we see that  the closure of $B^{\hat g}_\eps(x)$ is a subset of $\hat B_\pi(0_p)$.
Since $\vert K_{\hat g}\vert\leq 1$, by Klingenberg's Lemma
(cf.~\cite{Pet}, p.~182) there exists a geodesic loop  $c$
centered at 
$x$ of length $2\eps$ possibly not closing up smoothly.
Since the angle between a Killing field and a geodesic does not change,
the loop must close up smoothly. Here we have used that on a simply connected, locally 
homogeneous space $(M^n,g)$ for each point $p$ there exist Killing vector 
fields on $M^n$ spanning $T_p M^n$ (see Theorem 1 in \cite{No}).

Now the injectivity radius is a continuous function at $x$ (see Chapter VIII, Theorem 7.3 in \cite{KN2}). As a consequence, for 
small positive $t$ there exist closed geodesics $c_t$ centered at $(1-t)\cdot x$ of
length $2\eps+\delta(t)$, with $\lim_{t \to 0}\delta(t)=0$. We claim that in fact $\delta \equiv 0$.
To this end, we pick a sequence $(t_i)_{i \in \N}$ converging to zero, such that
the closed geodesics $c_{t_i}$ converge to a limit geodesic $\tilde c$ of length $2\eps$.
Since the metric $\hat g$ is real analytic, on compact sets the length functional
has only finitely many critical values by \cite{SS}. Consequently for large $i$ the 
length of $c_{t_i}$ must be $2\eps$.  
We can now conclude that also at the center point $0_p$ the injectivity radius
is less than or equal to $\eps < \pi$. This is a contradiction.
\end{proof}

In the next lemma we show that local isometries can be extended to balls of
uniform size.

\begin{lem}\label{lem:extend}
Let  $(\hat B_\pi(0_p),\hat g)$ be a geometric model, $x,y \in \hat B_\pi(0_p)$, and let 
\[
  \delta:=\delta(x,y):={\rm min}\{ i(\Vert x\Vert),i(\Vert y\Vert)\}\,.
\]
Then the open distance balls $B_\delta^{\hat g}(x)$ and $B_\delta^{\hat g}(y)$ are isometric.
\end{lem}

\begin{proof}
Since $(\hat B_\pi(0_p),\hat g)$ is locally homogeneous, there exists $\eps>0$ and
an isometry $h:B_\eps^{\hat g}(x) \to B_\eps^{\hat g}(y)$. We need to show that we can
extend this isometry to an isometry $H:B_\delta^{\hat g}(x) \to B_\delta^{\hat g}(y)$.

Since by Lemma \ref{lem:analytic} the locally homogeneous space $(\hat B_\pi(0_p),\hat g)$
is real analytic, it follows 
as in the proof of Proposition 10.5 in \cite{He} that both
the exponential maps ${\exp}_x \! \!: \hat B_\delta(0_x)\to B_\delta^{\hat g}(x)$
and ${\exp}_y \! \!: \hat B_\delta(0_y)\to B_\delta^{\hat g}(y)$ are real analytic diffeomorphisms
by Lemma \ref{lem:injest}. We set 
$$
  H :  B_\delta^{\hat g}(x) \to B_\delta^{\hat g}(y) \,\,;\,\,\,
   z\mapsto \big(\exp_y \circ \left({\rm D} h \right)_{x} \circ \left(\exp_x \right)^{-1}\big)(z)\,.
$$
The map $H$ coincides with $h$ on $B_{\eps}^{\hat g}(x)$. Moreover $H$ is analytic.
As a consequence, the analytic tensors $\hat g\vert_{B_\delta^{\hat g}(x)}$ and 
${H}^*( \hat g\vert_{B_\delta^{\hat g}(y)})$ 
are equal on $B_{\eps}^{\hat g}(x)$. By analyticity they coincide on all of 
$B_\delta^{\hat g}(x)$, hence $H$ is the desired extension. Clearly, $H$ is a diffeomorphism.
\end{proof}

\begin{rem}\label{volumeremark}
Note that the previous two lemmata  imply that there exist positive
constants $v_n \leq V_n $ such that
\begin{eqnarray}
 v_n r^n  \leq  \vol(B^{\hat g}_r(x) ) \leq V_n r^n \label{secondvolest}
\end{eqnarray}
for all $x \in \hat B_{\pi}(0_p)$ and all $0<r<\pi -|x|$. 
\end{rem}

For a Riemannian metric $g$ we denote by $\Ric(g)$ its Ricci tensor.

\begin{thm}\label{thm:homconv}
Let $(\hat B_\pi(0_{p_i}),\hat g_i)_{i \in \N} $ be a sequence of locally homogeneous geometric models.
Suppose that $\Ric(\hat g_i) \to 0$ for $i\to \infty$. Then, there exists a subsequence converging in 
$C^{1,\alpha}$-topology to a smooth flat limit 
space $(X,\hat g)$.
\end{thm}

\begin{proof}
Since for geometric models we have 
${\rm inj}_{\hat g_i}(0_{p_i})=\pi$ and $\vert K_{\hat g_i}\vert \leq 1$,
by the Cheeger-Gromov-compact\-ness theorem we may assume that $(\hat B_\pi(0),\hat g_i,0)_{i \in \N}$
converges to a $C^{1,\alpha}$-smooth manifold $(X,\hat g,x_0)$ in the pointed $C^{1,\alpha}$ topology.
Using  $\Ric(\hat g_i) \to 0$ as $i\to \infty$ this can be improved: $(X,\hat g)$ is smooth and satisfies
$\Ric(\hat g) = 0$ (see \cite{And}  and \cite{Pet}, Theorem 5.5).
The type of convergence  is described for example in 
Theorem \ref{thm:est}.

We claim now that this limit space $(X,\hat g)$ is locally homogeneous.
To this end, let $y \neq x_0$ be a point in $B^{\hat g}_{\pi-3\de}(x_0),$ and assume after pulling-back by the corresponding diffeomorphisms that convergence takes place in $B^{\hat g}_{\pi-\de}(x_0)$.
Then by Lemma \ref{lem:extend},
for all $i \in \N$ there exists a $\hat g_i$-isometry $f_i$ between $B_\delta^{\hat g_i}(x_0)$ and 
$B_\delta^{\hat g_i}(y)$. Choose a  countable
dense subset $S =\{s^k\}_{k \in \N}$ of $ B_\delta^{\hat g}(x_0)
\subseteq X$. 
The sequence $\left(f_i(s^k) \right)_{i\in \N}$
has a convergent subsequence with limit $f_\infty(s^k)\in
B_{\delta}^{\hat  g}(y)$. Using a Cantor diagonal procedure, one
can choose a subsequence $( f_{i_l})_{l \in \N}$
of $(f_{i})_{i \in \N}$, such that for any $k \in \N$
the sequence $(f_{i_l}(s^k))_{l\in \N}$ converges as $l$ goes to
infinity. 
This defines a one-Lipschitz map
$f_\infty:S \to X$, that is, $d_{\hat g}(f_{\infty}(s^j),f_{\infty}(s^k)  ) = d_{\hat g}(s^j, s^k)$.
Such a map can be extended to a distance preserving map from $B_\delta^{\hat g}(x_0)$
to $B_\delta^{\hat  g}(y)$. Clearly, this map is injective. Now, since $(X,\hat g)$ is smooth,
distance preserving maps are smooth as well (see for instance \cite[Theorem 11.1]{He}). Hence 
$f_\infty :B_\delta^{\hat  g}(x_0)\to B_\delta^{\hat  g}(y)$ is also surjective, and
this shows that these two balls are isometric. Consequently,
the limit space $(X, \hat  g)$ is locally homogeneous.

Finally, by \cite{Sp} locally homogeneous Ricci flat Riemannian manifolds are flat.
\end{proof}

\begin{proof}[Proof of the Gap Theorem]
As noticed in the introduction, it is sufficient to prove the equivalent estimate $\Vert \Rm(g) \Vert_g \leq C(n) \cdot  \Vert \Ric(g) \Vert_g$.
Let $V_n,v_n$ be the volume bounds mentioned in Remark
\ref{volumeremark},  $\eps(n)=
\eps(n,v_n,V_n) >0$ and $L(n) = L(n,v_n,V_n)$ be as in Theorem
\ref{thm:est}. W.l.o.g. $\ti \eps(n) := \frac{\eps(n)}{V_n3^n} \leq 1$.
%  if it is not then choose it in the proof of Theorem \ref{thm:est} to be so.

The proof goes by contradiction. 
Suppose that there exists a sequence $(M^n_i,g_i)$ of $n$-dimensional homogeneous
spaces with $\Vert \Rm(g_i)\Vert= {\ti \eps(n)  }^{2/n}\leq 1 $ and $\Vert \Ric(g_i)\Vert \to 0$ for $i\to \infty$.
Notice that $\Vert \Rm(g_i)\Vert \leq 1$ implies $\vert K_{g_i}\vert \leq 1$. Hence,
by the above discussion every such space has a geometric model $(\hat B_\pi(0_{p_i}),\hat g_i)$.
By Theorem \ref{thm:homconv} we may assume that the sequence 
$(\hat B_\pi(0_{p_i}),\hat g_i, 0_{p_i})_{i \in \N}$
converges to a flat limit space $(X,\hat g,x_0)$.

On the other hand, this contradicts the estimates in Theorem \ref{thm:est}
as we will show now. The balls $B_3^{\hat g_i}(0)$
are compactly contained in $\hat B_\pi(0_{p_i})$ and satisfy the volume estimates of Remark \ref{volumeremark}. 
Clearly, $\Vert \Ric(\hat g_i)\Vert \to 0$
for $i \to \infty$. Moreover, by our normalization we have that
\begin{eqnarray}
             \int_{B_{3}^{\hat g_i}(0)} \Vert \Rm(\hat g_i)\Vert^{n/2} d\mu_{\hat g_i}
             = {\ti \eps(n)} \cdot {\rm vol}(B_3^{\hat g_i}(0)) \in
             \big[\eps(n)\tfrac{v_n}{V_n},\eps(n)\big]  \label{thecontra} 
\end{eqnarray}
in view of the volume estimates. Thus, we can apply Theorem
\ref{thm:est} and conclude that the convergence to $(X,\hat g, x_0)$
is in fact in the $W^{2,p}$-topology, for any $p>n/2$. Using this, together with the volume estimates \eqref{secondvolest} and the fact that the limit is flat we obtain
\begin{align*}
 	v_n^{1/p} \cdot \tilde \varepsilon(n)^{n/2} &= v_n^{1/p} \cdot \Vert \Rm(\hat g_i)\Vert   \\
 		 & \leq {\rm vol}(B_1^{\hat g_i}(0))^{1/p} \cdot \Vert \Rm(\hat g_i)\Vert   \\
 		 & = \Big( \int_{B_{1}^{\hat g_i}(0)}  \Vert \Rm(\hat g_i)\Vert^{p} \Big)^{1/p}  \underset{i\to\infty}\longrightarrow 0,
\end{align*}
a contradiction.
\end{proof}

% From Theorem \ref{thm:est}, H\"olders inequality, the volume estimates
% \eqref{volest}, and the Theorem of \cite{Sp}, we see that  
% \begin{eqnarray}
% \lim_{i\to \infty}  v^{1/p} \Vert \Rm(\hat
%         g_i)\Vert && \leq \lim_{i\to \infty}  \Big( \int_{B_{1}^{\hat g_i}(0)}  \Vert \Rm(\hat
%         g_i)\Vert^{p} \Big)^{1/p}\cr
% && = \Big( \int_{B_{1}^{\hat g}(0)}  \Vert \Rm(\hat
%         g)\Vert^{p} \Big)^{1/p}\cr
% && = 0
% \end{eqnarray}
% for any $p> \frac{n}{2}$, which, using the  volume estimates
% \eqref{volest} again,
% contradicts \eqref{thecontra}.
% \end{proof}

%%%%%%%%%%%%%%%%%%%%%%%%%%%%%%%%%%%%%%%%%%%%%%%%%%%%%%%%%%%%%%%%%%%%%%%%%%%%%%%%%%%%%%%%%%%%%%
%%%%%%%%%%%%%%%%%%%%%%%%%%%%%%%%%%%%%%%%%%%%%%%%%%%%%%%%%%%%%%%%%%%%%%%%%%%%%%%%%%%%%%%%%%%%%%%

\section{A weak convergence result}\label{sec:conv}

In this section we will state a convergence result and curvature estimates for
Riemannian manifolds $(D^n,g)$ without boundary (possibly incomplete)
satisfying certain integral curvature 
and volume bounds. These results were applied to  geometric models of homogeneous
spaces in the proof of the Gap Theorem. Notice though that in this section
no symmetry assumptions are made on $(D^n,g)$ whatsoever.

We start by defining the $W^{1,p}$ harmonic radius of $(D^n, g)$. 
Given a chart 
$$
  \psi:V \to \psi(V)=Y\subset \R^n\,,
$$
$V \subset D^n$,  we denote by 
$$
  ( g_{jk}^\psi)_{1\leq j,k \leq n}:Y\to \R
$$
the coordinate functions of $g$ in the chart $\psi$. Then we set
 \begin{eqnarray}
\norm D  g\norm_{L^{p}}(\psi)
  := 
   \Big( \int_{\psi(V)} \sum_{j,k,l=1}^n |\partial_l  g_{jk}^\psi|^{p}  dy  \Big)^{\frac{1}{p}},
\end{eqnarray} 
where $dy$ refers to Lebesgue measure on $\R^n$ and $\partial_l g_{jk}^\psi $ 
refers to the standard Euclidean partial derivative in the $l$-th direction of the function 
$g_{jk}^\psi$. 

Clearly the above norm depends on the choice of the chart. For instance if $\psi(V)$ is
the Euclidean standard ball $B_1^{\rm std}(0)$ of radius one and
if $\psi_r := d_r \circ \psi$, $d_r(y)=r\cdot y$ is the 
dilation by factor $r>0$, then
$$
  \norm D g\norm_{L^{p}}(\psi_r)=r^{\frac{n}{p}-3}\cdot \norm D g\norm_{L^{p}}(\psi)\,.
$$
We assume now that distance ball $B^g_1(y)$ is compactly embedded in $D^n$.
Furthermore, we assume that there exists constants $0<v<V$ such that
for all $B_r^g(x)  \subseteq B_{1}^g(y)$ and all $0\leq r \leq 1$ we have
\begin{eqnarray}
          v r^n \leq \vol(B_r^g(x)) \leq V r^n\,.\label{volest}
\end{eqnarray}

\begin{defin}[Harmonic radius]\label{harmonic}
Let $(D^n,g)$ be as above and let $0< p<\infty$. Then
the $p$-harmonic radius $r_{{\rm har},p}^g(x)$ 
at a point $x \in B_1^{g}(y)$ is the supremum of all $r > 0$ 
with the following property:
There exists a $C^\infty$-smooth chart 
$$
  \psi_r=(\psi_r^1,...,\psi_r^n): V \to  B_r ^{\rm std}(0)
 $$ 
around $x$ with $\psi_r(x)=0$, $V\subseteq B_{1}^g(y) $ with the following properties:
\begin{itemize}
\item[(i)]  $ \tfrac{1}{2} \de_{jk} < g_{jk}^{\psi_r} < 2
  \de_{jk}$, and $g_{ij}(0) = \de_{ij}$.
\item[(ii)] $r^{1-\frac{n}{p}}\cdot \norm D  g \norm_{L^{p}}(\psi_r) <2$.
\item[(iii)] The map $\psi_r$ is harmonic, that is $\lap_{g} \psi^m_r = 0$  
   for all $m \in \{1,\ldots,n\} $.
\end{itemize}
\end{defin}

Note, that nowhere in the definition do we require that
$V=\psi_r^{-1}(B_r(0))$ be a geodesic ball. All three conditions are
invariant  under
the simultaneous scaling $(g,r)\mapsto (\lambda^2\cdot g,\lambda \cdot r)$
for some $\lambda >0$, where then the ball $B_1^g(y)$ has
to be replaced by $B_{1\lambda}^{\lambda^2\cdot g}(y)$ in the above definition.
As a consequence, the harmonic radius scales as a radius.

The  proof given here follows essentially the proof given in Appendix B of \cite{Sim}, which, as explained there, essentially follows the
proof of Main Lemma 2.2 of \cite{And} (see Remark 2.1 there), using some notions  
coming from \cite{AnCh}.

\begin{thm}\label{thm:har}
Let  $0<v \leq V $ and  $ p \in (n/2, \infty)$ be fixed constants.
Then there exist  \mbox{$\eps=\eps(v,V,n,p)>0$} and $L= L(v,V,n,p)>0$ 
such that the following holds.
Let $(D^n,g)$ be a smooth Riemannian
manifold without boundary, $y \in D^n$, such that $B_1^g(y)$  is compactly
contained in $D^n$. Assume that the volume estimates of {\rm (\ref{volest})}
are satisfied, and
\begin{equation}\label{littlericci}
  \int_{B_1^g(y)} \Vert\Ric(g)\Vert^{p} d\mu_{g} \leq 1,
\end{equation}
and \begin{equation}\label{RLpbound}
  \int_{ B_{1}^g(y)  } \Vert \Rm(g) \Vert^{n/2} d\mu_{g} \leq \epsilon\,.
\end{equation}
Then for all $x \in B_s^g(y)$, $s <1$ we have
\begin{eqnarray}
r_{{\rm har},2p}^g(x) &\geq & L (1-s).
\end{eqnarray}
\end{thm}
\begin{proof}

We are going to prove that 
 $r_{{\rm har},2p}^g(x) \geq  L \dist_g(x,\partial B_1^g(y))$ for all $x \in B_s^g(y)$.
Assume that the result is false.
We start with the more difficult case $p \in (n/2,n)$.
 Then there exists a sequence $(D^n_i,g_i,y_i)_{i \in \N}$
of pointed Riemannian manifolds without boundary,
 such that $B_1^{g_i}(y_i)$ is compactly contained in $(D^n_i,g_i)$
with the following properties: We have 
\begin{eqnarray}
  \int_{B_1^{g_i}(y_i)} \Vert \Rm(g_i) \Vert ^{n/2} d\mu_{g_i}\leq \tfrac{1}{i}\label{Rconz}
\end{eqnarray}  
and there exist $\tilde x_i \in B_1^{g_i}(y_i)$ such that
$$
  h_i(\tilde x_i)=
  \tfrac{r_{{\rm har},2p}^{g_i}(\tilde x_i)}{ \dist_{g_i}(\tilde x_i, \partial(B_{1}^{g_i}(y_i))) } \to 0 \quad \textrm{as}\quad i \to \infty\,.
$$  
Since $r_{{\rm har},2p}^{g_i}$ is lower semi-continuous  on
$\overline{B_1^{g_i}(y_i)}$ ($2p>n$),
there exist points
$x_i \in B_1^{g_i}(y_i)$ with $h_i(x_i) \leq h_i(x)$ for all $x\in B_1^{g_i}(y_i)$.
Clearly $\lim_{i \to \infty} h_i(x_i)=0$.
This in turn implies $\lim_{i\to \infty}
r_{{\rm har},2p}^{g_i}(x_i)\to 0$.  We deduce that
$$
  \mu_i:=\tfrac{1}{ r_{{\rm har},2p }^{g_i}(x_i)}\to \infty\quad \textrm{as}\quad i \to \infty\,.
$$
Next we consider the rescaled metrics $\tilde g_i:=\mu_i^2 \cdot g_i$.
Notice that the functions $h_i$, defined above, are invariant under such a scaling,
if we replace the ball of radius one by the balls $B_{\mu_i}^{\tilde g_i}(y_i)$ 
in the definition of harmonic radius.

We also have $r_{{\rm har},2p}^{\tilde g_i}(x_i)=1$, 
\begin{eqnarray}
\int_{B_{\mu_i}^{\tilde g_i}(y_i) }
\Vert \Ric(\tilde g_i)\Vert^{p} d\mu_{\tilde g_i}\to 
  0 \quad \textrm{as}\quad i \to \infty\,, \label{intestimatesRic}
\end{eqnarray}
and
\begin{eqnarray}
 \int_{B_{\mu_i}^{\tilde g_i}(y_i) }
  \Vert \Rm(\tilde g_i)\Vert^{n/2} d\mu_{\tilde g_i}\to 
  0 \quad \textrm{as}\quad i \to \infty\,, \label{intestimatesR}
\end{eqnarray}
since $p>\tfrac{n}{2}$ and the inequality (\ref{Rconz}) holds; notice that the integral
in (\ref{intestimatesR}) is invariant under scaling of the metric. 
Hence by the above mentioned scale invariance of the functions $h_i$ 
and the choice of points $x_i$ we have
for all $x \in B_{\mu_i}^{\tilde g_i}(y_i)$ 
$$
  \tfrac{r_{{\rm har},2p}^{\tilde g_i}(x)}{ \dist_{\tilde g_i}(x, \partial(B_{\mu_i}^{\tilde g_i}(y_i))) }
  \geq 
    \tfrac{r_{{\rm har},2p}^{\tilde g_i}(x_i)}{ \dist_{\tilde g_i}(x_i, 
    \partial(B_{\mu_i}^{\tilde g_i}(y_i))     }
   =  \tfrac{1}{ \dist_{\tilde g_i}(x_i,\partial(B_{\mu_i}^{\tilde g_i}(y_i))     }\to 
   0 \quad \textrm{as}\quad i \to \infty\,.
$$
Next, as in \cite{Sim} a simple triangle inequality estimate shows that for any $\rho>0$
and any  $x \in B_{\rho}^{\tilde g_i}(x_i)$ for $i \geq N(\rho)$ large enough we have
$$
   r_{{\rm har},2p}^{\tilde g_i}(x) \geq \tfrac{1}{2}\,.
$$
For ease of reading,  we remove the tildes from $\tilde g_i$
and write again $g_i$.

We set 
$   r_0 :=\tfrac{1}{100}$. Using the volume estimates and (i) we find harmonic coordinate charts
$\psi_i^s:U_i^s \to B_{r_0}^{\rm std}(0)$ with $\psi_i^s(x_i^s)=0$, $s=1,...,N=N(v,V,\rho,n)$,
such that  the  sets $(U_i^s)_{s=1}^N$ cover $B_\rho^{g_i}(x_i)$, and  their
intersection number is bounded from above by $Z(v,V,n)$. 
We write
$$
 \varphi_i^s:B_{r_0}^{\rm std}(0) \to U_i^s\,;\,\,y \mapsto (\psi_i^s)^{-1}(y)
$$
and call these maps charts also. We proceed as in \cite{Sim}
(cf.~\cite{Pet}): There exists a limit space $(X,d_X,x_\infty)$
of the sequence $(D^n_i,d(g_i),x_i)$ in pointed Gromov-Hausdorff-topology by 
Theorem 7.4.15 in \cite{BBI} in view of the volume \mbox{estimates.}
Arguing exactly as in \cite{Pet} after the proof of Fact 4 and at 
the beginning of Fact 5, we see first that 
$X$ is a $C^0$ manifold, with coordinate charts
\mbox{$\varphi_r:B_{r_0}^{\rm std}(0) \to V_r $}, and their construction implies the
following:
if $\varphi_t(B_{2\ep}^{\rm std}(v)) \subseteq V_r \cap V_t$, then the maps
$(\alpha_{i})_r^t:= (\varphi_i^r)^{-1}\of \varphi_i^t:B_{\ep}^{\rm std}(v) \to \R^n$ subconverge with respect to the $C^0$ norm to the maps $\alpha_r^t:=   (\varphi^r)^{-1}\of \varphi^t: B_{\ep}^{\rm std}(v) \to \R^n$ as $i \to \infty$.

Next we show that the limit space $X$ is a $C^{2,\beta}$-manifold.
We have to prove that the transition functions
\[
   (\alpha_i)_{\tilde s}^s:=(\varphi_i^{\tilde s})^{-1}\circ \varphi_i^{s}
\]
have a convergent subsequence in $C^{2,\beta}$. We sketch the
argument.

We can assume, choosing $i$ large enough, that $ (\alpha_i)_{\tilde
  s}^{ s}$ is defined on a small ball $ B_{2\delta}^{\rm std}(z) \subseteq B_{r_0}^{\rm std}(0)  $,
independent of $i$. The indices $s$ and $\ti s$ are fixed for the moment.
 By assumption (i) and (ii)
we have $W^{1,2p}( B_{r_0}^{\rm std}(0) )$-bounds for
$g_{jk}^{\psi_i^s}$, with $2p>n$.
By Morrey's Embedding theorem (see Theorem 7.17 in \cite{GiTr}) 
we obtain $C^\alpha(B_{r_0}^{\rm std}(0))$-bounds for $g_{jk}^{\psi_i^s}$ 
for some $0< \alpha < 1$  and by the Arzela-Ascoli-Theorem we
obtain, after taking a  subsequence, $g^{\psi_i^s}  \to h^s$  in
$C^{\al}(B_{r_0}^{\rm std}(0))$ as $i \to \infty$ for some Riemannian metric $h^s \in
C^{\al}(B_{r_0}^{\rm std}(0))$: we use the fact that $ \frac{1}{2}\de
\leq g^{\psi_i^s} \leq 2 \de$ freely, sometimes without explicit mention,
where $\de$ denotes the standard metric on $\R^n$.
It is well  known,
see remark B.2 in \cite{Sim}, that the transition functions for harmonic
coordinates satisfy 
$$
   \sum_{j,k=1}^n
    g^{jk}_{\psi_i^s}\cdot  \partial _j\partial _k
    ((\alpha_i)_{\tilde s}^s)^m =0
$$
for all $m = 1,...,n$ on  $B_{2\delta}^{\rm std}(z)$. 

Now by Schauder theory for the above elliptic differential equation one obtains
bounds for  $((\alpha_i)_{\tilde s}^s)^m$ in
$C^{2,\alpha}(B_{\delta}^{\rm std}(z)) $. 
By the Arzela-Ascoli-Theorem we obtain
a subsequence, which converges to the limit transition function
$(\alpha)_{\tilde s}^s = (\varphi^{\tilde s})^{-1}\of \varphi^s$ in
$C^{2,\beta}( B_{\delta}^{\rm std}(z) )$-topology for $\beta <
\alpha$.  Hence $X$  is a $C^{2,\beta}$ smooth manifold.
Also, pushing the metric $h^s$ back to $X$ using
$\varphi^s$ , we obtain a well defined $C^{\al}$
metric $h$ on $X$.

The bulk of the rest of the argument is devoted to showing that the limit manifold is
a smooth flat Riemannian manifold $(X,h)$ having  Euclidean volume
growth and hence is isometric to standard Euclidean space. This is used to obtain a
contradiction to the fact that the harmonic radii of the approximating
sequence is bounded from above.

We  sketch the proof given in \cite{Sim}.
We fix an $s$ and the maps $\varphi^s:B_{r_0}^{\rm std}(0) \to V_s$, 
$\varphi_i^s:B_{r_0}^{\rm std}(0) \to V_s$ associated to this $s$.
By (i) and (ii) in the definition of
harmonic radius we have $W^{1,2p}(B_{r_0}^{\rm std}(0))$-bounds for the metric
coefficients $g_i^s=g^{\psi^s_i}$ where $2p \in (n,2n)$. 
Of course this gives bounds in $W^{1,q}(B_{r_0}^{\rm std}(0))$ for all $q<n$,
since \mbox{$2p >n$}.
Also,  $g_i^s$ subconverges to $h^s$  in $L^{q}(B_{r_0}^{\rm std}(0)) $ for all $q \in (0,\infty)$
as $i\to \infty$,
since $g_i^s$ subconverges to $h^s$ in $C^{\al}(B_{r_0}^{\rm std}(0)).$
It is well known, that   in harmonic coordinates one has
\begin{eqnarray}
 g^{ab}\partial_a\partial_b g_{jk}= (g^{-1}*g^{-1}*D g *D g)_{jk}
   -2\Ric(g)_{jk}\, ,\label{richar}
\end{eqnarray}
see the reference in \cite{Sim}.
Remembering that  $s$ is fixed, we now use the notation $g_i = g_i^s$.
We know by (\ref{intestimatesRic}) that $\Ric(g_i)$ converges to zero
in $L^p(B_{r_0}^{\rm std}(0))$ as $i \to \infty$,
and by  H\"older's inequality the other terms are bounded in
$L^p(B_{r_0}^{\rm std}(0))$ in view of the fact that  $\partial g$
is bounded in $L^{2p}$, and  $g,g^{-1}$ are
bounded by $2 \de$. 
Now by the $L^p$-theory we obtain bounds for $g_i$ in
$W^{2,p}(K)$ (\cite{GiTr}, Theorem 9.11) for any compact $K \subseteq
 B_{r_0}^{\rm std}(0)$,
 and hence by the Rellich-Kondrachov-Embedding theorem, $g_i$ converges to 
$h$ strongly in $W^{1,q}(K)$,  for any  $q < p^*$ where 
$p^* = \frac{np}{n-p} >2p$. Note that $q =2p$ is a
valid choice here.

Next we are going to show that $(g_i)_{i \in \N}$ is
a Cauchy-sequence in $W^{2,p}(K)$, which gives us
$g_i \to h^s$ in $W^{2,p}(K)$.
For simplicity we write $g=g_i$ and $\tilde g=g_{\tilde i},$ $h = h^s$ in the next paragraph:
$s$ is still fixed. 
We set
\begin{eqnarray}
  L_{jk}&:=&
       (h^{ab}-g^{ab})\partial_a\partial_b g_{jk}
       -(h^{ab}-\tilde g^{ab})\partial_a\partial_b \tilde g_{jk} -2\Ric(g)_{jk}+2\Ric(\tilde g)_{jk}\cr
    &&   
     +(g^{-1}*g^{-1}*D g *D g)_{jk}
       - (\tilde g^{-1}*\tilde g^{-1}*D \tilde g *D \tilde g)_{jk}
       \,. \label{theleq}
\end{eqnarray}
Then we have
\begin{eqnarray}
   L_{jk}= h^{ab}\partial_a\partial_b (g-\tilde g)_{jk}\,.\label{opL}
\end{eqnarray}   
We are going to show that $L_{jk}$ becomes as small in $L^{p}(K)$
as we like for $i,\tilde i$ large enough.
The first term on the right of \eqref{theleq} may be estimated as follows
\begin{eqnarray*}
 \int_K \vert (h^{ab}-g^{ab})\partial_a\partial_b g_{jk} \vert^{p}
 &\leq &C(n)\cdot \sup_K| h-g| \int_K | \partial^2 g|^p.
\end{eqnarray*}
The integrals are taken with respect to the standard Lebesgue measure.
Since $g_i$ converges to $h$ in $C^{\alpha}(K)$ and 
 $g=g_i$ is bounded  in $W^{2,p}(K)$, as we showed above, we see that
 this term  is as small as we like in $L^{p}(K)$, as long as $i$ is large enough.
The second term can be estimated in a similar fashion.
The third and the fourth term converge by (\ref{intestimatesRic}) in
$L^p(K)$ to zero. The last two terms are dealt with as follows: $Dg_i \to
Dh$ and $Dg_{\ti i} \to Dh$ as $i, \ti i \to \infty,$ in $L^{2p}$, as
we explained above, and hence
$(g^{-1}*g^{-1}*D g *D g)_{jk}$ and $(\tilde g^{-1}*\tilde g^{-1}*D
\tilde g *D \tilde g)_{jk} $
converge to $(h^{-1}*h^{-1}*D h *D h)_{jk}$ in $L^{p}(K)$, which implies
that the sum of the last two terms of \eqref{theleq} converges to zero in
$L^p(K)$.

We deduce from \eqref{opL} and \eqref{theleq} that
$$
h^{ab}\partial_a\partial_b (g_i-g_{\tilde i})_{jk} = f(i,\tilde i)_{jk}
$$
with 
$$
   \int_K \vert f(i,\tilde i)\vert^{p} dy \leq \eps(i,\tilde i)
$$
and $\eps(i,\tilde i) \leq \eps$ for any $\eps>0$, if $i,\tilde i\geq N(\eps)$ are
large enough. Using again $L^p$-theory for elliptic operators we deduce 
by Theorem 9.11 in \cite{GiTr} that
$$
  \Vert g_i -g_{\tilde i}\Vert_{W^{2,p}(\tilde K)}
   \leq C_K \cdot \big( \Vert f(i,\tilde i)\Vert_{L^{p}(K)} + 
    \Vert g_i-g_{\tilde i} \Vert_{L^{p}(K)} \big)\,.
$$
on any smooth compact subset $ \ti K \subsub  K \subseteq B_{r_0}^{\rm std}(0)$.

This clearly implies that $(g_i)_{i \in \N}$
is a Cauchy sequence in $W^{2,p}(K)$ for compact subsets $K
\subseteq B_{r_0}^{\rm std}(0)$ , and consequently $g_i \to h$
for $i \to \infty$ in $W^{2,p}(K)\cap W^{1,q}(K),$ for any $q \in
[2p,p^*)$, in particular we can choose $q =2p$ here.

We come now back to the identity (\ref{richar}):  we can take the limit in $L^{p}(K)$ 
and deduce
$$
  h^{ab}\partial_a\partial_b (h)_{jk}=- (h^{-1}* h^{-1}* Dh *Dh)_{jk},
$$
where the right hand  of this equation is in $L^r(K)$ with $r = \frac{n}{2} +
  \sigma$ for some $\frac{n}{2}>\sigma >0$, and hence $h \in W^{2,r}(K)$ by the
$L^r$ theory.
 The Sobolev Embedding
Theorem tells us that $h \in W^{1,\frac{rn}{n-r}}(K),$ and we note
that 
$\frac{rn}{n-r} = \frac{(n/2 + \sigma)n}{(n/2 - \sigma)} = n + 2\sigma \cdot
(\frac{n}{n/2 - \sigma})\geq n +4\sigma$, and hence  the right hand side
of the above equation is bounded in $L^{\frac{n}{2} +  2\sigma}$.

Iterating this argument, we get $h \in W^{2,r}$ for all $r\in
[1,\infty)$: we first choose $\sigma>0$ very small, and $N \in \N$ large so that
$\frac{n}{2} + (N-1) \sigma $ is as close as we like, but less than $n$.
Then the Sobolev Embedding Theorem in the $N$'th iteration gives us that the right
hand side is in $L^r(K)$. 
Now using the $L^2$ and the $L^r$ theory, we conclude, as in
\cite{Sim}, that $h$ is $C^{\infty}(K)$.

At this stage we know that the limit manifold $X$ is a
$C^{2,\beta}$-manifold, but that 
the metric $h$ is $C^\infty$-smooth in
our constructed harmonic coordinates. 
Using a similar but simpler  argument to the one used above to show that $h$ is
$C^{\infty}$, we see that the transition functions on $X$, and hence $X$, is
$C^{\infty}$: see the argument given in \cite{Sim} for example.

Next, we show that the limit space $(X,h)$ is flat. 
Let $s$ be fixed and the map
$\varphi_i^s: B_{r_0}^{\rm std}(0) \to U_i^s \subseteq X_i $ 
and  $\varphi^s:  B_{r_0}^{\rm std}(0) \to U^s$ in $X$ be as above. Since the
metrics $g_i = g_i^s$ converge to $h^s$ in
$W^{2,p}(K)\cap W^{1,2p}(K)\cap C^{0,\al}(K)$ in these coordinates,
for fixed $s\in \{1,...,N\}$ we have writing $h=h_s$, abusing notation slightly,
\begin{eqnarray*}
  \lefteqn{ \int_{K} \Vert \Rm(h)\Vert^{p} \,d\mu_{h} =}&& \\
    &=& 
     \int_{K} \Vert h^{-1} * h^{-1}*D^2 h 
      + h^{-1}* h^{-1}*D h*Dh \Vert^{p}\, d\mu_{h}\\
    &=&    
     \lim_{i\to \infty}
    \int_{K}
    \Vert (g_i^s)^{-1} * (g_i^s)^{-1}*D^2 g_i^s 
      + (g_i^s)^{-1}* (g_i^s)^{-1}*D g_i^s *Dg_i^s \Vert^{p}\, d\mu_{g_i^s}\\
    &\leq&   
     \lim_{i\to \infty} \int_{B_{r_0}^{\rm std}(0)}  \Vert \Rm(g_i)\Vert^{p} \,d\mu_{g_i}
     \stackrel{ (\ref{intestimatesR})}{=} 0\,.    
\end{eqnarray*}
This shows that the limit space $(X,h)$ is flat. $(X,h)$ has Euclidean
volume growth since
the $(B_\rho^{g_i}(x_i),g_i)$'s also do. 
This implies that $(X,h)$ is isometric to $(\R^n,\de)$.

As in the proof of the Main Lemma 2.2
in \cite{And} one obtains now a contradiction:
see \cite{Sim} for more details.

The case $p \geq n$ is similar but easier. Arguing as above, we get
(locally) $g_i
\in W^{2,p}(K)$ and hence 
$\partial g \to \partial h$ in $L^s(K)$ for all $s \in
[1,\infty)$, in view of the Rellich-Kondrachov embedding Theorem, and
hence, arguing as above, $g_i \to
h$ in $W^{2,p}(K) \cap W^{1,p}(K) $ for all smooth compact sets $K
\subseteq B_{r_0}^{\rm std}(0)$.
The rest of the argument is the same.
\end{proof}

\begin{thm}\label{thm:est}
Let  $0<v\leq V $ and $p\in(\frac{n}{2},\infty)$ be fixed and $\eps=\eps(v,V,n,p)>0$
and $L= L(v,V,n,p)$ be the constants from Theorem {\rm \ref{thm:har}} above.
Let $(D^n_i,g_i, y_i)_{i \in \N}$ be a sequence of pointed smooth Riemannian manifolds without
boundary such that $B_1^{g_i}(y_i) $  is compactly
contained in $D^n_i$ for all $i \in \N$, $y_i \in D^n_i$.
Assume that 
\begin{eqnarray*}
          v r^n \leq \vol(B_r^{g_i}(x)) \leq V r^n\,.
          %\label{vol}
\end{eqnarray*}
for all $r \leq 1 $, for all $B_r^{g_i}(x)  \subseteq B_1^{g_i}(y_i)$, and 
\begin{eqnarray*}
  \int_{ B_{1}^{g_i}(y_i)  } \Vert \Rm(g_i) \Vert^{n/2} d\mu_{g_i} \leq \eps, 
  %\label{Rassump}
\end{eqnarray*}
and
$$
  \int_{B_1^{g_i}(y_i)} \Vert\Ric(g_i)\Vert^{p} d\mu_{g_i} \to 0  
   \quad \textrm{as}\quad i \to \infty
$$ 
for some $p\in(\frac{n}{2},\infty)$.
Then for all $z \in B_{s}^{g_i}(y_i)$, $s<1$, the $2p$-harmonic radius is
bigger than $L(1-s)$.
Furthermore, we find a smooth Ricci flat limit space $(X,g,x_0)$ in the following sense.
For all $s<1$  $B_{s}^g(x_0)$ is compactly contained in $X$
and there exist smooth diffeomorphisms
$F_i:B^{g}_{s}(x_0) \to
  (F_i(B^{g_i}_{s}(y_i)) \subseteq B^{ g_i}_{1}(y_i)$
  with $F_i(x_0) = y_i$, 
  such that $(F_i)^*( g_i) \to g$ in 
$W^{2,p}((B^{g}_{s}(x_0)) \cap W^{1,2p}((B^{g}_{s}(x_0))$, after
taking a subsequence. Also the map $(F_i)^{-1} :(B^{g_i}_{\si}(m_i),d_{g_i})
\to  (B^{g}_{2\si}(m),d_g)$, for any $m_i $ with $F_i^{-1}(m_i)
= m$ is an $\ep(i)$ Gromov-Hausdorff
approximation, $\ep(i) \to 0$ as $i \to \infty$, if $B^{g}_{8\si}(m) \subseteq B_{1}^g(s)(x_0)$:
 distance converges with respect to this map and
its inverse, as $i \to \infty$.
In particular, we have
\begin{eqnarray*}
  \lim_{i\to \infty} \int_{B_{s}^{g_i}(x)}  \Vert \Rm(g_i)\Vert^{p} d\mu_{g_i}
  =
   \int_{B_{s}^{g}(x)} \Vert \Rm(g)\Vert^{p}
   d\mu_{g} \,, 
   %\label{claimy}
\end{eqnarray*}
for all $s<1$.
\end{thm}

\begin{proof}
By assumption we may apply Theorem \ref{thm:har} to all the metrics $g_i$,
showing that the $2p$-harmonic radius on $B_{1}(y_i)$ is bounded
uniformly from below by $L(1-s)$, where $L$ is without loss of
generality less than $1/100$.
In these coordinates we have to establish the above claimed subconvergence.

Now after finding the contradiction subsequence $g_i$
in the proof of Theorem \ref{thm:har} we only used the estimate (\ref{intestimatesRic})
to construct the limit metric, which is then smooth. The estimate \eqref{RLpbound} was only 
used to prove the flatness of the limit metric.
Hence we may proceed precisely as above, replacing the $g_i$'s of the
proof of  Theorem \ref{thm:har}  by the $g_i$'s given in the statement
of the Theorem here, and we deduce convergence of the sequence
$(g_i)_{i\in \N}$
locally (in the harmonic coordinates from above) in the $W^{2,p}\cap
W^{1,2p}$-topology. The diffeomorphisms $F_i$ are
constructed using the transition functions, which we know converge
locally in $C^{2,\al}$, as explained in the proof of Theorem
\ref{thm:har}  above. This, combined with the
fact that $(g_i)_{i\in \N}$ converges to $h$ in $W^{2,p}\cap W^{1,2p}$  in the harmonic
coordinates from above, implies that
  $(F_i)^*( g_i) \to g$ in 
$W^{2,p}((B^{g}_{1-2\de}(x_0)) \cap
W^{1,2p}((B^{g}_{1-2\de}(x_0))$ as $i \to \infty$: 
see for example \mbox{Appendix B} in \cite{Sim} for
details on the  construction of the diffeomorphisms $F_i$ appearing in
the statement of this Theorem, and why the previous statement is true.
\end{proof}

%%%%%%%%%%%%%%%%%%%%%%%%%%%%%%%%%%%%%%%%%%%%%%%%%%%%%%%%%%%%%%%%%%%%%%%%%%%%%%%%%%%%%%%%555
%%%%%%%%%%%%%%%%%%%%%%%%%%%%%%%%%%%%%%%%%%%%%%%%%%%%%%%%%%%%%%%%%%%%%%%%%%%%%%%%%%%%%%%%%%%

\section{Curvature estimates}\label{sec:appl}

In this section we discuss the applications of the Gap Theorem to the study of homogeneous Ricci flows. The proofs of Theorem \ref{main:long} and Corollary \ref{main:cor} are presented. Moreover, we show that for homogeneous Ricci flows the Ricci curvature
satisfies  a \emph{time
doubling property}, similar to that of the full curvature tensor.

A Ricci flow solution is called homogeneous, if it is homogeneous at any time.
Notice, that it is sufficient to assume homogeneity at the initial time:
Homogeneity implies bounded curvature, hence for such initial metrics
there exists a unique solution with bounded curvature. As is well-known it 
then follows that isometries of the initial metrics will also be isometries
for all evolved metrics. Moreover, by \cite{Kot} the isometry group does not change over 
time.

By the Gap Theorem there exists $C(n)>0$, such that every $n$-di\-men\-sional homogeneous manifold $(M^n,g)$ satisfies
\[
	\Vert \Rm(g) \Vert \leq C(n) \cdot \Vert \Ric(g) \Vert.
\]
Theorem \ref{main:long} follows now directly from

\begin{thm}\label{Rmleqscal}
 If $\left(M^n, g(t)\right)_{t\in[a,b]}$ is a homogeneous Ricci flow solution then 
 \begin{equation}\label{Rmestimate}
 	\Vert \Rm(g(b)) \Vert \leq  \max \big\{ \tfrac1{8 (b-a)}\,, \,\, 16 \cdot C(n)^2 \cdot \big( \scal(g(b)) - \scal(g(a)) \big)\big\}.
 \end{equation}
\end{thm}
\begin{proof}
Let  $K := \Vert \Rm(g(b)) \Vert$. If $\frac1{8K} \geq b-a$ then $K \leq (8(b-a))^{-1}$
and the claim follows.
We are left with the case  $\frac1{8K} \leq b-a$.
The above bound implies that
\begin{align*}
	\int_a^b \Vert \Rm(g(t)) \Vert^2 dt  &\leq 
	\tfrac{C(n)^2}{2} \int_a^b 2\, \Vert \Ric(g(t)) \Vert^2 dt \\
	&= \tfrac{C(n)^2}{2} \int_a^b \scal(g(t))' dt \\
	&= \tfrac{C(n)^2}{2} \big(\scal(g(b)) - \scal(g(a))\big).
\end{align*}
Now if $t \in [b - \frac1{8K},b]$, then by the doubling time estimate we have
$\Vert \Rm(g(t)) \Vert \geq \unm \cdot \Vert \Rm(g(b)) \Vert$. Using
that $a \leq b - \frac1{8K}$, 
we deduce that
\[
	\int_a^b \Vert \Rm(g(t)) \Vert^2 dt \, \, \geq \, \, \int_{b-\frac1{8K}}^b \Vert \Rm(g(t)) \Vert^2 dt \,\, \geq \, \, \tfrac{K^2}{4}  \cdot \tfrac1{8K} = \tfrac{K}{32}\, ,
\]
and the theorem follows.
\end{proof}

Recall that along a homogeneous Ricci flow $(M^n,g(t))_{t\in [a,b]}$ the scalar curvature $s(t) := \scal(g(t))$ is a constant function on the manifold, hence 
\[
	s(t)' = 2 \cdot \Vert {\Ric(g(t)) } \Vert^2 \, \geq \, \tfrac{2}{n} \cdot s(t)^2\,.
\] 
If $s(t)$ does not vanish for $t\in [a,b]$ then by integrating we get
\begin{equation*}
% \label{eqn:estimatescal}
	% - \tfrac{s(b)}^{-1} +  {s(a)}^{-1}  \geq \tfrac{2}{n} \cdot (b-a).
	- \tfrac1{s(b)} + \tfrac1{s(a)}  \geq \tfrac{2}{n} \cdot (b-a),
\end{equation*}
Assuming that $s(a) > 0$, which implies $s(b) > 0$, one gets that 
\begin{equation}\label{eqn:estimatescalpos}
	 s(a)  \leq  \tfrac{n}{2} \cdot {(b-a)}^{-1},
\end{equation}
and on the other hand if $s(b)<0$ then also $s(a)<0$ and one obtains 
\begin{equation}\label{eqn:estimatescalneg}
	\vert s(b) \vert \leq  \tfrac{n}{2} \cdot {(b-a)}^{-1}.
\end{equation}
If in addition there exists a constant $C_1(n) > 0$ such that
\begin{equation}\label{eqn:estimatescalup}
	s(t)'  \, \leq \, C_1(n) \cdot s(t)^2\,
\end{equation}
then one also gets the reversed inequality 
\begin{equation}\label{eqn:estimatereverse}
% \label{eqn:estimatescal}
	% - {s(b)}^{-1} +  {s(a)}^{-1}  \leq C_1(n) \cdot (b-a).
	- \tfrac1{s(b)} + \tfrac1{s(a)} \leq C_1(n) \cdot (b-a).
\end{equation}
% then in case  $s(a) > 0$ one deduces also
% \begin{equation}\label{eqn:estimatescalbound}
%  \frac{1}{C} \cdot \frac1{\tfrac{1}{C\scal(b)}+b-a} \leq s(a) \leq  \tfrac{n}{2} \cdot \tfrac1{b-a},
% \end{equation}
% it is easy to see by a standard comparison argument that in each of these three situations $s(t)$ satisfies the corresponding Type I or Type III estimate, and this completes the proof.
% {\alert Please provide those estimates needed below!}

\begin{proof}[Proof of Corollary \ref{main:cor}]
Let $(M^n,g(t)_{ t \in [0,T)})$ be a homogeneous Ricci flow solution with finite extinction 
time $T < \infty$ and $\scal(g(0)) = 1$. Taking $a=0, b = t$ in \eqref{Rmestimate} yields
\[
	\Vert \Rm(g(t)) \Vert \leq  \max \big\{ \tfrac1{8 t}\,, \,\, 16 \cdot C(n)^2 \cdot \big( \scal(g(t)) - 1 \big)\big\}.
\]
Observe that the first expression on the right hand side is decreasing, whereas the second is increasing and diverges to $+\infty$ as $t\to T$ since the curvature blows up at $T$. Therefore there exists a unique $t_0 \in (0, T)$ such that 
\[
	\tfrac1{8 \, t} \leq 16 \cdot C(n)^2 \cdot (\scal(g(t)) -1)
\] 
for all $t \in [t_0,T)$ with equality at $t_0$.
Hence, on  $[t_0, T)$ we have the estimate
% \[
% 	\Vert \Rm(g(t)) \Vert \leq 16 \cdot C(n)^2 \cdot (\scal(g(t)) - \scal(g(0))),
% \]
% since the curvature blows up at $T$ and thus the scalar curvature must become positive and also blow up at $T$
% by the above. Since by assumption $\scal(g(0)) \geq 0$ this gives
\begin{equation}\label{eqn:estimateRmscal2}
	\Vert \Rm(g(t)) \Vert \leq C_2(n) \cdot \scal(g(t)),
\end{equation}
The upper bound stated in the first statement of Corollary \ref{main:cor} now follows by taking $a=t > t_0, b\to T$ in \eqref{eqn:estimatescalpos}.
The lower bound is well-known and follows immediately from the doubling time estimate:
see Lemma 6.1 in \cite{CLN}.

It remains to prove an upper bound for $t_0$ in terms of $T$. First notice that since $\scal(g(0)) = 1$, from \eqref{eqn:estimatescalpos} for $a=0$, $b\to T$ we have that $t_0 < T \leq n/2$, which in turn gives a uniform upper bound 
\[
	\scal(g(t_0)) \cdot t_0 \leq C_3(n)
\] 
by the very definition of $t_0$.
On the other hand, the estimate \eqref{eqn:estimateRmscal2} gives an upper bound for $\scal{'}$ as in \eqref{eqn:estimatescalup}.  Thus from \eqref{eqn:estimatereverse} for $a=t_0$, $b\to T$ we get
\[
	 \tfrac{t_0}{C_3(n)} \leq   \tfrac1{\scal(g(t_0))} \leq C_1(n) \cdot (T-t_0)\,,
\]
and this implies that $t_0 \leq \delta(n) \cdot T$ for some uniform $\delta(n) \in (0,1)$.

If the solution is immortal, i.e.~ defined for all $t\in [0,\infty)$, we take $a=t, b=2t$ in \eqref{Rmestimate} and obtain
\[
	\Vert \Rm(g(t)) \Vert \leq \max \big\{\tfrac1{8 t}, -16 \cdot C(n)^2 \cdot \scal(g(t)) \big\}.
\]
Taking $a=0, b=t$ in \eqref{eqn:estimatescalneg} yields the desired estimate. 

If the solution is ancient, i.e.~ defined for $t\in (-\infty,-1]$, we
take $a \to -\infty$, $b = t < -1$ in \eqref{Rmestimate}. This gives us
\begin{equation}\label{eqn:estimateRmscal}
	\Vert \Rm(g(t)) \Vert \leq C_2(n) \cdot \scal(g(t))\,,
\end{equation}
since $\scal(g(a)) \to 0$ as $a\to -\infty$ by (\ref{eqn:estimatescalpos}). 
The upper bound now follows by taking $b=-1, a=t\leq -1$ in \eqref{eqn:estimatescalpos}. To prove the lower bound, notice that from \eqref{eqn:estimateRmscal} we immediately get an upper bound for the evolution of scalar curvature as in \eqref{eqn:estimatescalup}. Thus we can apply \eqref{eqn:estimatereverse} for $b = -1$, $a = t < -1$, and this finishes the proof.
% {\alert Proof of lower bound}
\end{proof}

\begin{rem}\label{rem_ass}
The assumptions on the scalar curvature in Corollary \ref{main:cor} are not restrictive for non-flat solutions. Indeed, for the finite extinction time case, it follows from \cite{La},
and also from Theorem \ref{Rmleqscal}, that the scalar curvature blows up at the extinction time. Regarding immortal solutions, it is well-known that the scalar curvature cannot be positive, since this would imply finite extinction time. And if it vanishes, the solution is Ricci flat and hence flat. One can argue analogously for ancient solutions. In each case one can then scale the initial metric so that the assumptions are satisfied. 
\end{rem}

The following example shows, that even along homogeneous Ricci flow solutions
with positive scalar curvature the norm of the curvature tensor can decrease 
tremendously.

\begin{lem}\label{lem_s3}
On $S^3$ there exist a sequence $((g_l(t))_{t\in [0,T(l))})_{l \in \N}$ 
of homogeneous Ricci flow solutions with  $\Vert \Rm(g_l(0))\Vert=1$,
$\scal(g_l(0))>0$ and $T(l) \to \infty$ for $l\to \infty$.
\end{lem}

\begin{proof} On $S^3={\rm SU}(2)$ left-invariant metrics are in one-to-one 
correspondence with scalar product on the Lie algebra $\su(2)=T_e{\rm SU}(2)$.
Let $Q(X,Y)=\frac{1}{2} \tr (X \cdot Y^*)$ and consider
for $x_1,x_2,x_3>0$ the left invariant metrics $g=g(x_1,x_2,x_3)$, given by
$$
   g= x_1 \cdot Q\vert_{\Lm_1} \perp  
        x_2 \cdot Q\vert_{\Lm_2} \perp 
        x_3 \cdot Q\vert_{\Lm_3}\,.
$$  
Here $\Lm_1$ is spanned by the diagonal matrix with entries $\pm i$, 
$\Lm_2$ is spanned by the real skew-symmetric matrices and
$\Lm_3$ is spanned by the symmetric matrix having $i$ at the off-diagonal entries.
To compute the diagonal entries $r_1,r_2,r_3$ of the Ricci endomorphism $\Ric(g)$, see (\ref{eqn_Ricend}), of the metrics $g$ with respect to
the decomposition $\su(2)=\Lm_1\oplus \Lm_2 \oplus \Lm_3$ we use the formula (\ref{ricci}).
 This yields 
\begin{eqnarray*}
  r_1   &=& \tfrac{2}{x_1x_2x_3}\cdot (x_1^2-(x_2-x_3)^2) \\
r_2 &=& \tfrac{2}{x_1x_2x_3}\cdot (x_2^2-(x_1-x_3)^2) \\               
r_3 &=&\tfrac{2}{x_1x_2x_3}\cdot (x_3^2-(x_1-x_2)^2)           \,.
\end{eqnarray*}
Here we have used that  $b_1=b_2=b_3=8$ (see  \cite{WZ1}, p.~583),
that $[123]=4$ and that of course $d_1=d_2=d_3=1$. 
As is well-known the off-diagonal entries
of the Ricci endomorphism vanish: see Chapter 1, Section 5 in \cite{CK}.

The Ricci flow equation for these metrics is
given by $x_i'=-2x_i\cdot r_i$, $i=1,2,3$ and the volume normalized
Ricci flow by $x_i'=-2x_i\cdot r_i^0$, $r_i^0=r_i-\tfrac{1}{3}(r_1+r_2+r_3)$.
Recall that after a reparametrization in space and time the Ricci flow
and the volume normalized Ricci flow are equivalent.
Solutions to the normalized Ricci flow will be denoted by 
$(\bar g(\bar t))_{\bar t \in [0,\bar T)}$.

It is now convenient to introduce
new coordinates $\alpha =\tfrac{x_2}{x_1}$ and $\beta=\tfrac{x_3}{x_1}$.
Notice that 
$\tfrac{\alpha'}{\alpha}=\tfrac{x_2'}{x_2}-\tfrac{x_1'}{x_1}=2(r_1-r_2)$
and $ \tfrac{\beta'}{\beta}=2(r_1-r_3)$. The volume constraint $x_1x_2x_3=1$
reads now $(\alpha\beta)^\frac{1}{3}=\tfrac{1}{x_1}$. Therefore, the volume normalized Ricci flow, fixing volume one, is equivalent to
\begin{eqnarray*}
  \alpha' &=&
    \tfrac{8}{(\alpha\beta)^\frac{2}{3}}\cdot \alpha \cdot (1-\alpha)\cdot (1+\alpha-\beta)\\
   \beta' &=&
    \tfrac{8}{(\alpha\beta)^\frac{2}{3}}\cdot \beta \cdot (1-\beta)\cdot (1+\beta-\alpha)\,.  
\end{eqnarray*}
The sets $\{\alpha \equiv 1\}$, $\{\beta \equiv 1\}$ and
$\{\alpha \equiv \beta\}$ are invariant under this ordinary differential 
equation and $(1,1)$
is the unique zero in the domain $\{\alpha,\beta>0\}$.

For $l \in \N$, $l\geq 100$, we choose initial values 
$\alpha_0^l:=\tfrac{l}{4}+\frac{\sqrt{l-1}}{2}$
and $\beta_0^l:=\tfrac{l}{4}-\frac{\sqrt{l-1}}{2}$. Clearly $\alpha_0^l,\beta_0^l >1$
and $1+\alpha_0^l-\beta_0^l>0$, whereas $1+\beta_0^l-\alpha_0^l<0$.
That is, for the solution $(\alpha(t),\beta(t))$ with initial values $(\alpha_0^l,\beta_0^l)$
at time $\bar t=0$
we have $\alpha'<0$ and $\beta'>0$ until this solution reaches the line
$\{1+\beta-\alpha\equiv 0\}$ at a time $\bar t_0^l>0$. Notice that $\beta(\bar t_0^l)>\beta_0^l$
of course and that the time $\bar t_0^l$ is unique.

The initial values or chosen  such that $\scal(\bar g_l(0))=0$ and 
$\Vert \Ric(\bar g_l(0))\Vert^2 \geq c_1 \cdot l^{\frac{1}{3}}$ for $c_1>0$ independent of $l$. Another computation shows that at time $\bar t_0^l$ we have
$\scal((\bar g_l(\bar t_0^l))^2=\Vert \Ric(\bar g_l(\bar t_0^l))\Vert^2 \leq c_2 \cdot l^{-\frac{2}{3}}$ for a constant
$c_2$, independent of $l$.

Next, let $(g(t))_{t \in [0,T)}$ denote a solution to the unnormalized Ricci flow
and set $V(t):= \sqrt{x_1(t)x_2(t)x_3(t)}$. Notice that 
up to a constant this equals to the volume of $(S^3,g(t))$. As is well-known, but
also follows from the above equations, we have $V'(t)=-V(t)\cdot s(t)$,
where $s(t):=\scal(g(t))$ and consequently $V'(t)=-V^{\frac{1}{3}}(t)\cdot \bar s(t)$,
with $\bar s(t)=-V^{\frac{2}{3}}(t)\cdot s(t)$. Note, that the function
$(x_1x_2x_3)^{\frac{1}{3}}\cdot
\scal(g(x_1,x_2,x_3))$ is scale invariant.

Let $(g_l(t))_{t \in [0,T(l))}$ denote the solution to the unnormalized Ricci flow 
with initial value $g_l(0)$ satisfying $V(0)=1$,
$\alpha(0)=\alpha_0^l$ and $\beta(0)=\beta_0^l$.
By the above we know that $\Vert \Ric(g_l(0))\Vert \geq  \sqrt{c_1}\cdot l^{\frac{1}{6}}$.
Next, let us denote by $t_0^l\in (0,T(l))$ the unique time with $1+\beta(t_0^l)-\alpha(t_0^l)=0$. By the above we have
$\bar s(t) \leq \bar s(t_0^l)\leq \epsilon(l):=\sqrt{c_2} \cdot l^{-\frac{1}{3}}$ 
for all $t \in [0,t_0^l]$, since the scalar curvature is still
increasing along the volume normalized Ricci flow.  
Using $V'(t)=-V^{\frac{1}{3}}(t)\cdot \bar s(t)$
and $V(0)=1$ we deduce
$V(t)\geq (1-\epsilon(l)\cdot t)^{\frac{3}{2}}$ for 
$t \leq \max\{t_0^l,\tfrac{1}{\epsilon(l)}\}$.
Suppose now $T(l)\leq C$, for all $l \geq 100$ and a constant $C>0$.
Clearly 	 this implies $t_0^l < C$ and hence there exists $l_0 >0$
such that $V(t_0^l) \geq 0.5$ for all $l\geq l_0$. Since $
\Vert \Ric(\bar g_l(\bar t_0^l))\Vert=V^{\frac{1}{3}}(t_0^l)\cdot \Vert \Ric(g(t_0^l))\Vert
\leq \epsilon(l)$, we conclude on the other hand side 
by the doubling time property of the Ricci flow that the extinction time of
these solutions cannot be uniformly bounded. Contradiction.

Finally notice that if $(g(t))_{t \in [0,T)}$ is a
Ricci flow solution then for any $\lambda>0$ also
$(\lambda\cdot g(\tfrac{t}{\lambda}))_{t \in [0,\lambda\cdot T)}$ is 
a solution. As a consequence the term $\Vert \Rm(g(0))\Vert_{g(0)} \cdot T$ is
invariant under parabolic rescaling. Then choosing $t_l>0$ as close to zero as 
we like as our new initial time and performing the parabolic rescaling just described
shows the claim.
\end{proof}

The next application is a doubling time estimate for the Ricci curvature along homogeneous Ricci flows. It shows that the Ricci curvature cannot grow too quickly. 

\begin{prop}
% [doubling time for Ricci curvature]
Let $(M^n,g(t))_{t\in [0,b]}$ be a homogeneous Ricci flow solution. If $\Vert \Ric(g(0)) \Vert = 1$, then
\[
	\Vert \Ric(g(t)) \Vert \leq 2, 
\]
for all $0\leq t \leq 1/C(n)$.
\end{prop}

\begin{proof}
According to the evolution equation for $\Vert \Ric \Vert^2$ along Ricci flow \cite{CLN}, Lemma 2.40, 
together with the Gap Theorem, one has that
\[
	\frac{\rm d}{{\rm d} t} \Vert \Ric \Vert^2 \leq C_1 (n) \cdot \Vert \Rm \Vert \cdot \Vert \Ric \Vert^2 \leq C(n) \cdot \Vert \Ric \Vert^3\,.
\]
Recall that in the homogeneous case ${\Vert \Ric \Vert^2}$ is a constant function. By a standard comparison argument, one obtains $\Vert \Ric(g(t)) \Vert \leq \tfrac2{2- C(n) \cdot t}$
since $\rho(t) = \tfrac2{2- C(n) \cdot t}$ is the solution to $\tfrac{\rm d}{{\rm d} t} \rho^2 = C(n)  \cdot \rho^3$, $\rho(0) = 1$. The proposition now follows.
\end{proof}

%%%%%%%%%%%%%%%%%%%%%%%%%%%%%%%%%%%%%%%%%%%%%%%%%%%%%%%%%%%%%%%%%%%%%%%%%%%%%%%%%%%%%%%%555
%%%%%%%%%%%%%%%%%%%%%%%%%%%%%%%%%%%%%%%%%%%%%%%%%%%%%%%%%%%%%%%%%%%%%%%%%%%%%%%%%%%%%%%%%%%

\section{Non-collapsed homogeneous ancient solutions}\label{sec:comancient}

In this section we prove Theorem \ref{thm:comp} which is essentially Theorem \ref{main:comp}, and we show in Lemma \ref{lem:ancient} that for any unstable homogeneous Einstein metric there exists a non-collapsed homogeneous ancient solution emanating from it.

Since non-trivial ancient solutions  $(g(t))_{t \in (-\infty,0]}$
to the Ricci flow have positive scalar curvature,
non-trivial ancient  homogeneous  solutions must develop a \mbox{Type I} singularity 
close to their extinction time  by Corollary \ref{main:cor}.  By \cite{Na} and \cite{EMT}
the blow-up of such a solution will subconverge to a non-flat homogeneous gradient shrinking soliton. These homogeneous limit solitons where classified by \cite{PW}. 
Up to finite coverings they are the Riemannian product
of a compact homogeneous Einstein space and a flat factor. Notice that
the flat factor might be absent.

Also by Corollary \ref{main:cor}, ancient homogeneous solutions
 develop a \mbox{Type I} behavior in the past. It is then natural to consider
the corresponding blow-downs
$$ 
   g_i(t):=\tfrac{1}{s_i}\cdot g(s_i\cdot t)
$$
for a sequence $\{s_i\}_{i \in \N}$ with $s_i \to \infty$ and all 
$t \in (-\infty,0]$. In the non-collapsed situation,
it follows by \cite{Na} and \cite{CZ} that the sequence
$(g_i(t))_{ i \in \N}$ subconverges to a non-flat asymptotic soliton.

A compact homogeneous space has a presentation $M^n=G/H$,
where $G$ is a compact Lie group acting transitively on $M^n$ with compact isotropy 
group $H$.
Notice that $G$ and $H$ are not necessarily connected. Since $G$ is  compact, 
there exists an ${\rm Ad}(G)$-invariant scalar product $Q$ on the
Lie algebra $\Lg$ of $G$. Let $\Lm$ denote the $Q$-orthogonal complement of
$\Lh$ in $\Lg$. Then
the set $\mathcal{M}^G$ of $G$-homogeneous metrics on $G/H$
can be viewed as the set of ${\rm Ad}(H)$-invariant scalar products on $\Lm$.
This set in turn can be viewed as the Euclidean space $S^2(\Lm)^{{\rm Ad}(H)}$ 
of symmetric, positive-definite, ${\rm Ad}(H)$-equivariant linear 
endomorphisms of $\Lm$ as follows $ g(x,y) = Q\vert_{\Lm}(g\cdot x,y)$,
where $x,y \in \Lm$. Recall, that a $G$-homogeneous Einstein metric on $G/H$
is a critical point of the total scalar curvature functional 
$$
  \TS: \mathcal{M}^G_1 \to \R \,\,;\,\,g \mapsto \scal(g)
  $$
restricted to the space $\mathcal{M}^G_1$ of $G$-homogeneous metrics of volume one,
and that the gradient flow 
$$
   g'(t) = -2\cdot g(t)\cdot \Ric_0(g(t))
$$
of $\TS$ is nothing but the volume-normalized Ricci flow for $G$-homogeneous metrics.
Here we consider the Ricci-endomorphism $\Ric(g(t))$, defined by
\begin{eqnarray}
   \Ric(g(t))(x,y) = g(t)(\Ric(g(t))\cdot x,y)\label{eqn_Ricend}
\end{eqnarray}
and then also $g(t)$ as an endomorphism as above.
Since the space of $G$-homogeneous metrics is finite-dimensional, we have existence and uniqueness of homogeneous Ricci flow solutions also backwards in time.

\begin{lem}\label{lem:analyticflow}
The gradient flow of $\TS$ on $\mathcal{M}^G_1$ is analytic.
\end{lem}

\begin{proof}
The set  $\mathcal{M}^G_1$ is an algebraic subvariety of 
the Euclidean space  $\mathcal{M}^G$, and at the same time a smooth submanifold. Moreover, when
fixing a $Q$-orthonormal basis of $\Lm$ any $G$-homogeneous metric
can be considered an ${\rm Ad}(H)$-equivariant matrix.
Now the Ricci endomorphism $\Ric(g)$ of 
$g \in \mathcal{M}^G$ can be written down explicitly as in Proposition 1.5
in \cite{BWZ}. It is a rational map and consequently the same is
true for its traceless part $\Ric_0(g)$. This shows the claim.
\end{proof}

\begin{thm}\label{thm:comp}
Let $G/H$ be a compact homogeneous space.
Then any  non-collapsed homogeneous ancient solution has a unique compact 
asymptotic soliton, which is a homogeneous Einstein metric on $G/H$. 
\end{thm}

\begin{proof}
As mentioned above, by \cite{Na} and \cite{GZ}
for any sequence of blow-downs $(g_i(t))_{t \in (-\infty,0]}$ 
of a non-collapsed ancient solution $(g(t))_{t \in (-\infty,0]}$ there exists a non-flat 
asymptotic soliton to which they subconverge. 
Now by \cite{PW} we know, that up to a finite covering this asymptotic soliton is a product 
of a compact homogeneous Einstein space and a flat factor.
We have to exclude the flat factor. 

If there were a flat factor $\R^k$, then for large $i$ the volume of the metrics
$g_i(-1)$ would be unbounded, whereas $\scal(g_i(-1))\to c_\infty >0$
for $i \to \infty$. As a consequence, for the 
unit volume normalization $\bar g(t)$ of $g(t)$ the 
scalar curvature would be unbounded as $t\to -\infty$. 
This is a contradiction, since the scalar curvature is increasing along the volume-normalized Ricci 
flow on a compact homogeneous space.

Thus, the asymptotic soliton is compact. Next, we claim, that along
the backward volume normalized homogeneous Ricci flow the scalar curvature
of the solution $\bar g(\tau)$, $\tau = -t$, cannot converge to zero.
This is clear, since otherwise the asymptotic limit soliton would be flat.

It follows, that for any sequence $(\tau_i)_{ i \in \N}$
converging to $+\infty$, we will find a subsequence, such that
$\bar g(\tau_{i_j})$ is a Palais-Smale-sequence $C$ for 
$\TS$ in $\{\TS \geq \eps\}$ for some $\eps>0$. Then
by Theorem A in \cite{BWZ} there exists a subsequence, which 
converges to a homogeneous limit metric $g_\infty$ on $G/H$.

Uniqueness of the limit follows from Lemma \ref{lem:analytic}, since it is well-known that if the $\omega$-limit set of an analytic gradient flow solution is non-empty, then it consists of a single point, see for instance \cite{Loj}.
\end{proof}

\begin{rem}\label{rem_noncoll}
If a compact homogeneous space $G/H$, with $G,H$ connected, 
is not a homogeneous torus bundle, that is if there exists no compact 
intermediate subgroup $H<K<G$ such that $K/H$ is a torus,
then any ancient solution on $G/H$ is non-collapsed.
This is seen as follows. For a collapsed solution to the backward volume-normalized Ricci flow the scalar curvature must tend to zero. That means that on that space there exists a zero-Palais-Smale-sequence. Now by Theorem 2.1 in \cite{BWZ} the claim follows.
\end{rem}

Next, we show that for any unstable homogeneous Einstein metric there
are always ancient solutions emanating from it. Recall that a homogeneous Einstein metric is called \emph{unstable}, if it is not a local maximum of $\TS$.

Recall that the Milnor fibre of a critical point $g_E$ is defined as follows
$$
   F_+(g_E) := \{ g \in \mathcal{M}^G_1 : \TS(g)= \TS(g_E)+r^N\} \cap B_r(g_E)
$$
where $B_r(g_E)$ is a ball in $ \mathcal{M}^G_1$ of a very small radius $r>0$ and $N$ is large. Clearly, it is empty if and only if $g_E$ is a local maximum of $\TS$.

If an Einstein metric $g_E$ is non-degenerate and unstable, then
by the unstable manifold theorem there exists solutions to the
gradient flow of $\TS$ emanating from it. Since the scalar curvature is
still increasing along such flow lines, it follows that these solutions
have positive scalar curvature. Hence, by \cite{La} they are ancient.
By standard Morse theory the Milnor fibre is homotopy equivalent to $S^{k-1}$, 
$k$ being the dimension of the positive eigenspace of the Hessian of 
$\TS$ at the critical point $g_E$. 
In general however, the unstable Einstein metric $g_E$ might be degenerate
and even not isolated.

\begin{lem}\label{lem:ancient}
Let $g_E$ be a $G$-homogeneous unstable Einstein metric on a compact
homogeneous space $G/H$. Then there exists a $G$-homogeneous ancient
solution emanating from $g_E$. Moreover, the dimension of such solutions
can be estimated from the below by the cohomological dimension of the 
Milnor fibre of $\TS$ at $g_E$.
\end{lem}

\begin{proof}
This follows from  \cite{NS}. It is shown there,
that if the Milnor fibre is not empty, then there exists a solution 
to the negative gradient flow of $\TS$ with omega limit set $g_E$.
Moreover, the authors show that the Cech-Alexander cohomology groups 
of the set of such solutions and the Milnor fibre agree.
\end{proof}

We should point out that in the above lemma we do not claim that all these
ancient solutions are pairwise non-isometric.

\begin{rem}
An Einstein metric $g_E$ on a compact manifold is the Yamabe metric
in its conformal class. Moreover, by Theorem C in \cite{BWZ} any nearby metric
with constant scalar curvature is a  Yamabe metric as well, provided $g_E$
is not in the conformal class of the round metric. It follows, that
if a homogeneous Einstein metric is unstable, then it is not
a local maximum of the  Yamabe functional. Hence, by Theorem 1.6 in \cite{Kr} 
there exists an ancient solution emanating from it. It may be possible to
adjust the proof in \cite{Kr} to show that this ancient solution can also
be chosen to be homogeneous.

Let us also mention that there are general existence results for unstable homogeneous Einstein metrics on
compact homogeneous spaces
relying on min-max principles \cite{BWZ}, \cite{Bo2}, \cite{Gr}, to which
Lemma \ref{lem:ancient} can be applied. For instance, if a compact homogeneous space $G/H$ satisfies a certain algebraic property (its graph having two non-toral components, see \cite{BWZ} for
details), then there exists $C>0$ such that  $\TS^{-1}([C,\infty))$ is disconnected.
Moreover,  there is also a smooth path
of metrics in $\{0<\TS \leq C\}$ joining two connected components of
$\TS^{-1}([C,\infty))$.
By a standard mountain pass lemma the existence of a critical point of
$\TS$ follows: see Proposition 3.7 in \cite{BWZ}.
It also follows that there must exist a
critical point which is not
a local maximum of $\TS$, since by Proposition 1.5 in \cite{BWZ} the set
of all critical points of $\TS$ is a disjoint union of
finitely many compact, connected, semialgebraic sets on each one of
which $\TS$ is constant.

% There are general existence results for homogeneous Einstein metrics relying
% on a mountain pass lemma \cite{BWZ}, \cite{Bo2}, \cite{Gr}. 
% In its simplest version it means that a level set $\TS^{-1}(c)$
% is disconnected for some $c>0$ and that there exists a smooth path
% of metrics in $\{0<\TS \leq c\}$ joining two connected components of $\TS^{-1}(c)$.
% When applying the gradient flow of $\TS$ to this path, it must remain
% hanging at a critical point of energy less than $c$. 

% Using such variational methods, one might obtain a high-dimensional family of Einstein metrics.
% A priori, some of these Einstein metrics could even
% be local maxima of $\TS$. However, due to the mountain pass lemma, 
% there must exist at least one Einstein metric $g_E$, which is not a local maximum 
% of $\TS$. 
\end{rem}

%%%%%%%%%%%%%%%%%%%%%%%%%%%%%%%%%%%%%%%%%%%%%%%%%%%%%%%%%%%%%%%%%%%%%%%%%%%%
%%%%%%%%%%%%%%%%%%%%%%%%%%%%%%%%%%%%%%%%%%%%%%%%%%%%%%%%%%%%%%%%%%%%%%%%%%%%%
%%%%%%%%%%%%%%%%%%%%%%%%%%%%%%%%%%%%%%%%%%%%%%%%%%%%%%%%%%%%%%%%%%%%%%%%%%%%%%

\section{Examples of collapsed ancient solutions}\label{sec:exancient}

For collapsed homogeneous ancient solitons there will of course be no asymptotic gradient shrinking soliton in the category of smooth manifolds.
However, when working with Riemannian groupoids, as introduced by Lott in \cite{Lo2}, 
one might hope to prove the existence of a locally homogeneous asymptotic gradient shrinking
soliton. That means that after considering the blow-downs of a collapsed ancient solution, 
we pull back these metrics to a ball in the tangent space.
Due to the curvature estimates provided by Theorem \ref{main:long} these balls can
be chosen to have a uniform radius. The convergence in the category of Riemannian
groupoids then only means that these locally homogeneous metrics
converge in $C^\infty$-topology to a locally homogeneous limit metric on this fixed ball
(cf.~ Section \ref{sec:lochom}). In other words, one considers the convergence of the corresponding \emph{geometric models}, defined in Section \ref{sec:lochom}. Now an asymptotic soliton, 
would be a locally homogeneous product metric of an Einstein metric with positive Ricci 
curvature and a flat factor. In the
collapsed case the flat factor cannot be absent.
Notice, that if a collapsed ancient homogeneous solution
does admit a locally homogeneous asymptotic soliton,  it must be non-flat
by the curvature estimates provided in Corollary \ref{main:cor}. 

We would like to mention that by \cite{Rod} there exist locally homogeneous Einstein metrics 
on $(S^3\times S^3)/S^1_r$ of positive scalar curvature, which by \cite{Kow}
cannot be extended to globally homogeneous compact Einstein spaces  (cf.~\cite{BWZ}, p.~725).
Here $S^1_r$ is embedded into the maximal torus of $S^3\times S^3$ with
irrational slope $r$. Whether such locally homogeneous spaces occur as compact 
factors in asymptotic solitons is unknown.

In order to provide an example of a collapsed ancient solution with non-compact
singularity model, we cannot work with homogeneous spaces whose isotropy representation 
admits only two summands: see \cite{DK} for a classification. 
Instead, we are looking for homogeneous spaces
$G/H$ not admitting any $G$-invariant Einstein metric, to prevent a compact
singularity model, but which in addition are homogeneous $S^1$-bundles.

We recall, how to compute the Ricci curvature of a compact homogeneous space
$G/H$. As mentioned above, every $G$-invariant metric on $G/H$ is uniquely determined by an
$Ad(H)$-invariant scalar product on $\Lm$. It follows that
for any $G$-in\-variant metric $g$ on $G/H$ there exists a decomposition
$\Lm=\Lm_1 \oplus \dots \oplus \Lm_\ell$ of $\Lm$ into $Ad(H)$-irreducible summands, 
such that $g$ is diagonal with respect to $Q$, that is
\begin{eqnarray*}
  g = x_1 \cdot Q |_{\Lm_1}\perp \dots \perp x_\ell \cdot Q|_{\Lm_\ell}
\end{eqnarray*}
with $x_1,...,x_\ell>0$. By \cite{WZ2}, \cite{PaSa}, the diagonal entries $r_m$
of the Ricci endomorphism $\Ric(g)$ of $g$ are given by
\begin{eqnarray}
   r_m & =&
   \frac{b_m}{2x_m}
       -\frac{1}{2d_m}\sum_{j,k=1}^\ell \jkm \frac{x_k}{x_m x_j}
       +\frac{1}{4d_m}\sum_{j,k=1}^\ell \jkm \frac{x_m}{x_j x_k}\,.
       \label{ricci}
\end{eqnarray}
Here, $-B|_{\Lm_m} = b_m\cdot Q|_{\Lm_m}$
and $d_m = \dim \Lm_m$, where $B$ denotes the Killing form of $\Lg$. The structure
constants $[ijk]$ with respect to the above decomposition of $\Lm$ are
defined as follows:
\begin{eqnarray*}
  [ijk] = \sum Q([e_{\alpha},e_{\beta}],
   e_{\gamma})^2
\end{eqnarray*}
where the sum is taken
over $\{e_{\alpha}\}$, $\{e_{\beta}\}$, and
$\{e_{\gamma}\}$,
$Q$-orthonormal bases for $\Lm_i$, $\Lm_j$ and $\Lm_k$, respectively.
Notice that $[ijk]$ is invariant under permutation of $i,j,k$.

\begin{exa}\label{exa:4}
On $G/H=({\rm SU}(n){\rm SU}(n))/ (\Delta {\rm SU}(n-1)\Delta {\rm U}(1))  $
there exist for \mbox{$n\geq 3$} a one-parameter family of homogeneous ancient solutions
with the same non-compact asymptotic soliton $(E_-,g_-^1) \times \R$, where
$$
   E_-=({\rm SU}(n){\rm SU}(n))/(\Delta({\rm SU}(n-1)){\rm U}(1){\rm U}(1))\,,
$$
and the same non-compact singularity model $(E_+,g_+)\times  \R^{4(n-1)+1}$, where
$$
  E_+=({\rm SU}(n-1){\rm SU}(n-1))/\Delta {\rm SU}(n-1)\,.
$$
The Einstein metric $g_-^1$ on $E_-$ is the unstable homogeneous Einstein metric on $E_-$
and $(E_+,g_+)$ is a compact symmetric space.
Furthermore, there exists one further homogeneous ancient solution in the closure of
the above family with the asymptotic soliton $(E_-,g_-^2)$ and the same singularity model.
Here $g_-^2$ is now the stable homogeneous Einstein metric on $E_-$.
Finally,
all the above ancient solutions have positive Ricci curvature.
\end{exa}

\begin{proof}
We set $G=G_1G_2$, with $G_1=G_2={\rm SU}(n)$ 
and $H=\Delta {\rm SU}(n-1)\Delta {\rm U}(1)$ for $n \geq 3$. Here
$H$ is embedded into $G$ as follows:
Consider the subgroup ${\rm SU}(n-1)$
of ${\rm SU}(n)$, embedded as an upper $(n-1)\times (n-1)$-block.
 Then the semisimple part of $H$ is embedded diagonally in ${\rm SU}(n-1){\rm SU}(n-1) \subset G$.
The subgroup ${\rm SU}(n-1)$ commutes with its centralizer ${\rm U}(1)$
in ${\rm SU}(n)$. This ${\rm U}(1)$ is embedded diagonally into ${\rm SU}(n)$,
the first $(n-1)$ diagonal entries being equal. Now
$\Delta{\rm U}(1)$ is embedded diagonally into the product ${\rm U}(1) {\rm U}(1)$.

We choose the ${\rm Ad}(G)$-invariant scalar product
$Q(X,Y)=\frac{1}{2} \tr (X \cdot Y^*)$ on $\Lg$.
 Then $b_m=4n$ for all $m$ (see \cite{WZ1}, p.~583).
Next, let $\Lm'$ denote the orthogonal complement of
$\un(n-1)\oplus \un(1)$ in $\un(n)$, considered
as a subspace in the first factor $\su_1(n)$ and let 
$\Lm''$ be defined accordingly. We set $\Lm_1=\Lm'\oplus \Lm''$ and conclude
$d_1=\dim \Lm_1=4(n-1)$.
The space $\Lm_2$ is the orthogonal complement of $\Delta {\su}(n-1)$
in $\su_1(n-1)\oplus \su_2(n-1)$, thus $d_2=n(n-2)$.
Finally, $\Lm_3$ is the orthogonal complement 
of $\Delta {\un}(1)$ in $\un_1(1)\times \un_1(1)$, hence $d_3=1$. 

The group $G=G_1 G_2$ admit the involution
$f(g_1,g_2)= (g_2,g_1)$. Clearly, we have $f(H)=H$.
We set now $\hat G=  \Z_2 \ltimes G$ and $\hat H=\Z_2 \ltimes H$. 
Then $G/H=\hat G/\hat H$ as manifolds.
Moreover, the modules $\Lm_1$, $\Lm_2$ and $\Lm_3$ are
now $\Ad(\hat H)$-irreducible and of course inequivalent, since there dimensions
are different.

We will consider the three-parameter family of homogeneous metrics
$$
   g= x_1 \cdot Q\vert_{\Lm_1} \perp  
        x_2 \cdot Q\vert_{\Lm_2} \perp 
        x_3 \cdot Q\vert_{\Lm_3}
$$      
and compute their Ricci curvatures for $x_1,x_2,x_3>0$.

The only non vanishing structure constants are $[112]$ and $[113]$.
To this end notice that ${\rm U}(n)/({\rm U}(n-1){\rm U}(1))$ is a symmetric space,
hence $[111]=0$. Next, $\Lm_2 \oplus \Lm_3 \oplus \Lh$ is a subalgebra, and therefore
$[122]=[123]=[133]=0$. Since $({\rm SU}(n-1){\rm SU}(n-1))/\Delta({\rm SU}(n-1)$
is also a symmetric pair, we have $[222]=0$. Moreover,  $[333]=0$,
since $\Lm_3$ is an abelian subalgebra. Finally $[\Lm_2,\Lm_3]=0$. 

We have that $4n=[311]$ by the identity $d_3b_3=\sum_{i,j=1}^3[3ij]$ from Lemma 1.5 in \cite{WZ2}, 
since $d_3=1$. Next, we claim that $[211]=4d_2$. To this end we choose a standard orthonormal basis
of $\Lm_1$ consisting of $2(n-1)$ skew-symmetric elements with only two non-vanishing
entries $\pm 1$ and $2(n-1)$ symmetric elements with two non-vanishing entries $i$. 
For each of these basis vectors $e$ there
exist precisely $1+2(n-2)$ other basis elements not commuting with $e$. The first special
basis element $e^*$ has its non-vanishing entries at the same spot as $e$ does.
The Lie-bracket $[e,e^*]$ is diagonal and has two non-vanishing entries $\pm 2i$.
When computing the projection of $[e,e^*]$ onto $\su(n-1)$ one deduces that
$\Vert [e,e^*]_{\su(n-1)}\Vert^2 =2\frac{n-2}{n-1}$. The computation of the other
$2(n-2)$ non-vanishing brackets is standard and we obtain
$$
   [211]=4(n-1)\cdot \big( 2\tfrac{n-2}{n-1} +2(n-2) \big) \cdot \tfrac{1}{2} =4(n-2)n=4d_2\,,
$$
noticing that all the brackets in question had to be projected to the diagonally
embedded $\Delta \su(n-1)$ in $\su(n-1)\oplus \su(n-1)$.

Since the three modules $\Lm_1, \Lm_2$ and $\Lm_3$ are inequivalent,
it follows from Schur's Lemma that the Ricci tensor of the metric $g=g(x_1,x_2,x_3)$
is diagonal as well and by (\ref{ricci}) we deduce
\begin{eqnarray}
  r_1   &=& 
                   \tfrac{2n}{x_1}-\tfrac{n(n-2)}{2(n-1)}\cdot \tfrac{x_2}{x_1^2}
                   -\tfrac{n}{2(n-1)}\cdot \tfrac{x_3}{x_1^2} \label{eqn1}\\
r_2 &=&
                 \tfrac{2(n-1)}{x_2}+ \tfrac{x_2}{x_1^2} \label{eqn2}\\
r_3 &=&
            n \cdot \tfrac{x_3}{x_1^2} \label{eqn3}
\end{eqnarray}
Recall that the Ricci flow on the homogeneous space
$\hat G/\hat H$ is given by $x_i' =-2x_i\cdot r_i$, $i=1,2,3$.
Furthermore, we can reduce dimension by one considering
the volume normalized Ricci flow $x_i'=-2x_i \cdot r_i^0$,
where $r_i^0 =r_i -\frac{1}{N}\cdot \scal$ denotes the entries
of the traceless part of the Ricci endomorphism, $N =\dim \hat G/\hat H$.
 To understand this normalized
Ricci flow it is convenient to introduce new coordinates
$\alpha=\tfrac{x_2}{x_1}$ and $\beta=\tfrac{x_3}{x_1}$. We have
that 
$\tfrac{\alpha'}{\alpha}=\tfrac{x_2'}{x_2}-\tfrac{x_1'}{x_1}=2(r_1-r_2)$
and similarly $ \tfrac{\beta'}{\beta}=2(r_1-r_3)$. The volume constraint $x_1^{d_1}x_2^{d_2}x_3=1$
reads in these new coordinates $(\alpha^{d_2}\cdot \beta)^{\frac{1}{N}}=\frac{1}{x_1}$. 
\begin{figure}[b]
\begin{center}
\includegraphics[scale=0.6]{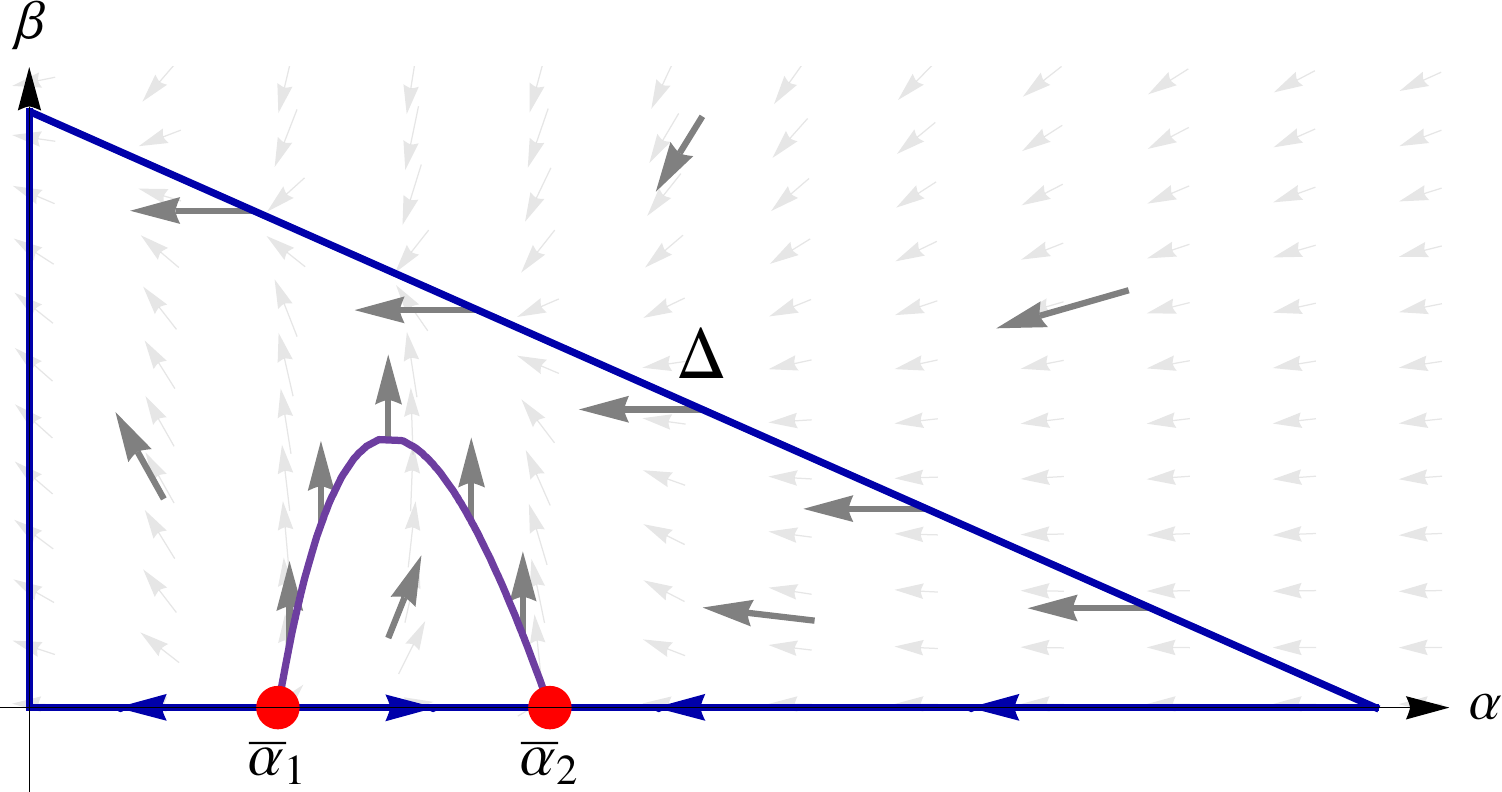}
\end{center}
\caption{Volume normalized Ricci flow on $\hat G/\hat H$ for $n=4$}
\label{pic:flowlines}
\end{figure}
We deduce
that the volume normalized Ricci flow is equivalent to 
\begin{eqnarray}
  \alpha' &=&
    p(\alpha,\beta)\cdot \Big( 4(n-1) -\tfrac{n^2-2}{n}\cdot \alpha
   -\tfrac{4(n-1)^2}{n}\cdot \tfrac{1}{\alpha}                                           -\beta\Big)\label{eqn4}\\
  \beta' &=&
     q(\alpha,\beta) \cdot \Big( 4(n-1) -(n-2)\cdot  \alpha
   -(2n-1)\cdot \beta
      \Big)\label{eqn5}\,,
\end{eqnarray}
where 
 $ p(\alpha,\beta)=\alpha \cdot \frac{n}{n-1}\cdot \alpha^{\tfrac{d_2}{N}}\cdot \beta^{\frac{1}{N}}$
 and
 $q(\alpha,\beta)=\beta \cdot \frac{n}{n-1}\cdot \alpha^{\tfrac{d_2}{N}}\cdot  \beta^{\frac{1}{N}}$.
Now, notice that if
we divide the right hand side of the system (\ref{eqn4}) and (\ref{eqn5}) by the positive function $\tfrac{n}{n-1} \cdot \alpha^{\tfrac{d_2}{N}}\cdot \beta^{\frac{1}{N}}$ only
the time-parametrization of solutions does change, but the integral curves do not.
Consequently, the volume normalized Ricci flow of $\hat G/\hat H$ is a up to time
reparametrization equivalent to
\begin{eqnarray}
  \alpha' &=&
     -\frac{4(n-1)^2}{n}+ 4(n-1)\cdot \alpha -\frac{n^2-2}{n}\cdot \alpha^2
     -\alpha \cdot \beta\label{eqn6}\\
  \beta' &=&
     \beta \cdot \big( 4(n-1) -(n-2)\cdot  \alpha
   -(2n-1)\cdot \beta
      \big)\label{eqn7}\,,
\end{eqnarray}
restricted to the domain $\{\alpha,\beta > 0\}$.

We turn to the qualitative behavior of this system, however on the slightly
larger domain $D:=\{\alpha > 0\} \cup \{ \beta \geq 0\}$ (cf.~Figure \ref{pic:flowlines}).
First notice,
that the positive $\alpha$-axis in invariant under this system.
In fact this restriction is up to reparametrization
nothing but the volume normalized Ricci flow 
on $E_-$. Moreover,
the above system admits precisely two constant solutions $(\bar \alpha_1,0)$
and $(\bar \alpha_2,0)$ with
$$
  \bar \alpha_1 =\tfrac{2(n-1)}{n+\sqrt{2}} \quad \textrm{and}\quad
   \bar \alpha_2 =\tfrac{2(n-1)}{n-\sqrt{2}}\,.
$$
To this end, for $\alpha>0$ the condition $\alpha'=0$ 
implies 
$$
 \beta =  -\tfrac{4(n-1)^2}{n}\cdot \tfrac1{\alpha}+ 4(n-1) -\tfrac{n^2-2}{n}\cdot \alpha\,.
$$
Plugging this into (\ref{eqn7}) yields a quadratic equation for $\alpha$, which does not
have real solutions. This shows in particular
that the space $\hat G/\hat H$ does not admit a $\hat G$-invariant Einstein metric.

Next, we consider the right hand side of the above system as a smooth vector field $X$.
Its differential $(DX)_{(\bar \alpha_i,0)}$, $i=1,2$, is upper triangular,
with eigenvalues
$$
 \lambda_1^i =(-1)^{i+1} \cdot \tfrac{4(n-1)\sqrt{2}}{n}
 \quad\textrm{and}\quad
  \lambda_2^i=4(n-1)-(n-2)\cdot \bar \alpha_i>0\,.
$$ 
This shows, that $(\bar \alpha_1,0)$ is a node, while $(\bar \alpha_2,0)$ is a saddle point, 
whose unstable manifold intersects the $\alpha$-axis transversally.

The  set $\{\beta ' \geq 0\} \cup \{\beta >0\}$ is a right triangle $\Delta$,
depicted in \mbox{Figure \ref{pic:flowlines}}.
Actually the hypotenuse belongs to $\Delta$, but the two other sides do not.
A computation shows now, that $X$ intersects its hypotenuse transversally pointing into its interior. 
It follows that a maximal solution in $D$, starting in $\Delta$, 
cannot leave $\Delta$.  Moreover, since in $\Delta$ we have
$r_1\geq r_3$ and since $r_2,r_3>0$, we conclude that any metric in  $\Delta$
has positive Ricci curvature.

As we saw above, $(\bar \alpha_1,0)$ is a node. Therefore, there exist a one-parameter family
of solutions in $D$, which emanate from it. Since these solutions cannot leave
$\Delta$, they have positive Ricci curvature, hence they are ancient  by \cite{La}.
Clearly, the compact factor of the asymptotic soliton of these ancient solutions 
is the unstable Einstein metric $g_-^1$ of $E_-$.

The second ancient solution is given by the unstable manifold of the
stable Einstein metric $(\bar \alpha_2,0)$ of the compact factor $E_-$.
It lies in the closure of the above one-parameter family of solutions,
but has an ancient soliton,  whose compact factor is isometric to $g_-^2$.
It is not hard to show, using the Ricci curvature formulas (\ref{eqn1}), (\ref{eqn2})
and (\ref{eqn3}),  that these solutions cannot be isometric.

For all these ancient solutions $(\alpha(t),\beta(t))_{t \in (-\infty,T)}\in D$
it remains to compute their singularity model. Let us mention at this point, that
they reach the $\beta$-axis. However this point does not correspond to a 
Riemannian metric on $\hat G/\hat H$. It is clear that
$\lim_{t \to T}= \alpha(t)=\bar \alpha=0$, while $\lim_{t\to T}\beta(t)=\bar \beta>0$.
We rescale now these metrics such that $x_2(t) \equiv 1$. Since 
$\bar \alpha=0$, we deduce $\lim_{t \to T} x_1(t)=\infty$ and since $\bar \beta >0$
it follows that also $\lim_{t \to T} x_3(t)=\infty$. As a consequence these metrics converge
to a limit product metric on the singularity model $E_+ \times \R^{4(n-1)+1}$.
Since $E_+$ is isotropy irreducible, the limit metric on $E_+$ is Einstein.
\end{proof}

%%%%%%%%%%%%%%%%%%%%%%%%%%%%%%%%%%%%%%%%%%%%%%%%%%%%%%%%%%%%%%%%%%%%%%%%%%%%%%%%%%%%%%%%555
%%%%%%%%%%%%%%%%%%%%%%%%%%%%%%%%%%%%%%%%%%%%%%%%%%%%%%%%%%%%%%%%%%%%%%%%%%%%%%%%%%%%%%%%%%%

\section{Bounds on the gap}\label{sec:ex}

According to the decomposition of the space of curvature operators into irreducible $\Or(n)$-modules 
in dimensions $n\geq 4$, the operator $\Rm$ decomposes as 
\[
  \Rm = \Rm_I + \Rm_{\Ric_0} + \Weyl\,,
\]
where $\Rm_I = \frac{\scal}{n (n-1)} \id \wedge \id $, $\Rm_{\Ric_0} = \frac2{n-2} \Ric_0 \wedge \id$, $\Ric_0 = \Ric - \frac{\scal}{n}  \id$ is the traceless Ricci tensor and $\Weyl$ is the Weyl tensor. A standard computation shows that
\begin{equation}\label{eqn:normcurvop}
    \left\Vert \Rm_I \right\Vert^2 = \frac{\scal^2}{2n(n-1)} 
    \quad\textrm{and}\quad
     \left\Vert \Rm_{\Ric_0}\right\Vert^2 = \frac{\Vert \Ric_0 \Vert^2}{n-2} .
\end{equation}
The aim of this section is to prove the following 

\begin{lem}\label{lem:gapboud}
For $n\geq 4$ there exist homogeneous spaces $(M^n,g)$ such that
\[
  \Vert \Weyl(g) \Vert_g \geq \sqrt{\tfrac{n-2}{n-3}}\cdot \Vert \Rm(g)\Vert_g\,.
\]
\end{lem}

\begin{proof}
An $n$-dimensional real Lie algebra $\Lg$ is called \emph{almost abelian} if it admits a codimension-one abelian ideal $\Ln$. In other words, there exist a basis $\{e_i\}_{i=0}^{n-1}$ for $\Lg$ and an endomorphism $A\in \gl_{n-1}(\R)$ such that the Lie bracket is given by
\[
    [e_0, e_i] = -[e_i,e_0] =  A \, e_i, \qquad [e_i,e_j] = 0, \qquad i,j \neq 0.
\]
 Let us consider an inner product
$\langle \, \cdot,\cdot \rangle$ on $\Lg$ that makes $\{e_i\}$ orthonormal, and denote by $(S_A,G)$ the corresponding simply-connected Lie group with left-invariant Riemannian metric. We denote by 
$D = \unm \left( A+ A^t\right) $ and $Q = \unm\left( A-A^t\right)$
the symmetric and skew-symmetric parts of $A$, respectively.
It follows from \cite{Meu} that the curvature operator 
$\Rm : \Lambda^2 \Lg \to \Lambda^2 \Lg$
is given by
\begin{align*}
    \Rm\left( e_0 \wedge e_i \right) &= - e_0 \wedge \left(D^2 + [D,Q] \right) e_i, \qquad i \neq 0, \\
    \Rm\left( e_i \wedge e_j \right) &= D e_j \wedge D e_i, \qquad i,j \neq 0\,,
\end{align*}
the Ricci endomorphism $\Ric : \Lg \to \Lg$ by
\begin{align*}
 	\Ric(e_0) = -\left(\trace D^2\right) e_0,\quad
	\Ric(e_i) = \left([Q,D] - \left( \trace D\right) D\right) e_i, \quad i\neq 0
\end{align*}
and the scalar curvature by $\scal = - \trace D^2 - \left(\trace D\right)^2$.
If in particular one takes $Q = 0$ and $D$ such that $\trace D = 0$, then 
\begin{align*}\label{ineqRicscal}
	\scal^2 = \Vert D\Vert^4, \quad \Vert \Ric_0 \Vert^2 = \frac{n-1}{n} \cdot \Vert D\Vert^4 \quad\textrm{and}\quad
	\Vert \Rm \Vert^2 = \unm\left(\Vert D \Vert^4 + \Vert D^2 \Vert^2 \right)\,.
\end{align*}

For getting a concrete estimate we choose $D$ diagonal, with eigenvalues $\lambda_1 = n-2$ and $\lambda_2 = \ldots = \lambda_{n-1} = -1$. After using \eqref{eqn:normcurvop} and a straightforward computation we obtain
\begin{align*}
    \left\Vert \Rm_I \right\Vert^2 + \left\Vert \Rm_{\Ric_0}\right\Vert^2 &=   
    	\tfrac12  (n-1)(n-2) (2n-3) \\
    	&\leq \frac{ (n-1)(n-2) (2n^2 - 8n + 9)}{2 (n-3)} = \frac1{n-3} \cdot \Vert \Rm \Vert^2,
\end{align*}
which shows the claim.
\end{proof}

%%%%%%%%%%%%%%%%%%%%%%%%%%%%%%%%%%%%%%%%%%%%%%%%%%%%%%%%%%%%%%%%%%%%%%%%%%%%%%%%%%%%%%%%555
%%%%%%%%%%%%%%%%%%%%%%%%%%%%%%%%%%%%%%%%%%%%%%%%%%%%%%%%%%%%%%%%%%%%%%%%%%%%%%%%%%%%%%%%%%%

%%%%%%%%%%%%%%%%%%%%%%%%%%%%%%%%%%%%%%%%%%%%%%%%%%%%%%%%%%%%%%%%%%%%%%%%%%%%%%%%%%%%%%%55

\end{document}